\documentclass{amsart}

\usepackage{color}

\usepackage{emptypage} 

\makeindex 

\pagebreak
\newtheorem{thm}{Theorem}

\newtheorem{lem}[thm]{Lemma}

\newtheorem*{theorem*}{Theorem}
\theoremstyle{definition}

\newtheorem{defn}[thm]{Definition}
\newtheorem{remk}[thm]{Remark}
\newtheorem{remks}[thm]{Remarks}

\newtheorem{exm}[thm]{Example}
\newtheorem{exms}[thm]{Examples}
\newtheorem{notat}[thm]{Notation}

{\hfill$\square$\end{defn}}
{\hfill$\square$\end{remk}}
{\hfill$\square$\end{remks}}
{\hfill$\square$\end{exm}}
{\hfill$\square$\end{exms}}
{\hfill$\square$\end{notat}}

\newcommand{\minus}{\scalebox{0.75}[1.0]{$-$}}

\renewcommand{\ker}{{\rm{ker}}}

\newcommand{\tuborg}{\left\{\begin{array}{ll}}
\newcommand{\sluttuborg}{\end{array}\right.}

\newcommand{\Tight}{\textit{Tight}}

\newcounter{elno}

\newcounter{elno-abc}   

\newcounter{elno-abc-prime}   

\usepackage[english]{babel}
\usepackage[utf8x]{inputenc}
\usepackage{amsmath}
\usepackage{graphicx}
\usepackage[colorinlistoftodos]{todonotes}

\title{Classification of tight contact structures on some Seifert fibered manifolds}
\author{Tanushree Shah}

\begin{document}
\begin{abstract}
We classify tight contact structures with zero Giroux torsion on some Seifert-fibered manifolds with four exceptional fibers. We get the lower bound by constructing contact structures using Legendrian surgery. We use convex surface theory to obtain the upper bound.
\end{abstract}

\maketitle

\section{Introduction} 

Contact geometry has been central to many advances in low-dimensional topology. For instance, it played a key role in Kronheimer and Mrowka's proof that all non-trivial knots satisfy property P (where non-trivial surgery yields a manifold with a non-trivial fundamental group) [32]. It also underpinned Ozsváth and Szabó's results showing that the unknot, trefoil, and figure-eight knots are uniquely determined by Dehn surgery [2, 37, 38], among other applications. Martinet proved in \cite{Mar} that every closed, oriented 3-manifold admits a contact structure, prompting the study and classification on a 3-manifold. Bennequin divided the class of contact structures on a 3-manifold into overtwisted and tight in \cite{Ben}. The overtwisted ones are in a 1-1 correspondence with homotopy classes of tangent planes on a 3-manifold, proved by Eliashberg in \cite{Eli}, establishing a link to the topology of the manifold. \

\ 

The emphasis shifted to tight contact structures, which turned out to be far more nuanced. Etnyre and Honda established the non-existence of tight contact structures on certain Seifert fibered manifolds \cite{EH}, though their existence remains an open problem in many cases. Classification efforts, when existence is known, primarily target prime atoroidal manifolds, as tight structures respect the connected sum decomposition of 3-manifolds, while essential tori can generate infinitely many non-isotopic tight structures. Consequently, initial classifications were often restricted to atoroidal manifolds. Tight contact structures have been classified on $S^3$ \cite{Eli03}, lens spaces \cite{Hon01}, and most Seifert fibered spaces with three exceptional fibers \cite{GLS, GLS01, Mat, Wu}.

\ 

The focus on toroidal manifold came with the classification of tight structures on 3-torus, by Kanda \cite{Kan} and Giroux \cite{Gir01} (independently) using the theory of characteristic foliations and convex surfaces. Recently, Simone used twisted Heegaard Floer contact invariant to classify tight structures on some toroidal plumbed 3-manifolds \cite{Sim}.\ 

\ 

We build on the results above to classify tight contact structures on some toroidal Seifert fiberd manifolds with four exceptional fibers with zero Giroux torsion. For the Seifert fibered manifold $M(g,e_0;p_1/q_1,\allowbreak...,p_4/q_4)$, where $g$ is the genus of the base surface $B$, and $e_0\in \mathbb{Z}$, $\frac{p_i}{q_i}\in (0,1)\cap \mathbb{Q}$. 
We denote the continued fraction expansion of $-\frac{q_i}{p_i}$ by $[a_0^i,a_1^i,...,a_{m_i}^i]$ with all $a_{i}<-1$ integers. The surgery diagram corresponding to $M(0,e_0;p_1/q_1,...,p_4/q_4)$ is shown in the Figure $\ref{lg}$.

\

 We first look at an example case in detail before proving the general result. We look at the tight contact structures on $M(0, -4;1/2,1/2,1/2,1/2)$. Once we have calculated the tight contact structures on this example case, it is computationally easy to generalise to manifolds $M=M(e_0,0;p_1/q_1,p_2/q_2,\allowbreak p_3/q_3,p_4/q_4)$ with $e_0(M)\leq-4.$  
\begin{thm}
\label{general thm}

 Let $M=M(e_0,0;p_1/q_1,p_2/q_2,\allowbreak p_3/q_3,p_4/q_4)$  where $e_0\leq -4$ and $\frac{p_i}{q_i}\in (0,1)\cap \mathbb{Q}$ and $gcd(p_i,q_i)=1$. On $M$ there are exactly $|(e_0(M)+1)\Pi_{i=1}^4\Pi_{j=1}^{m_i}(a_j^i+1)|$ tight contact structures with zero Giroux torsion up to contact isotopy. All of these can be constructed by Legendrian $-1$ surgery and hence are Stein fillable. For each $n\in\mathbb{Z}$ there exists at least one tight contact structure with n-Giroux torsion on $M$. These tight contact structures are not weakly fillable. 
\end{thm}

We start by constructing tight contact structures with zero Giroux torsion by Legendrian surgery to get a lower bound on the number of tight contact structures. These tight contact structures might be isotopic. To distinguish non isotopic contact structures we use Lisca-Mati\'c's result in \cite{LM} which uses Chern numbers and Stein structures. To get the upper bound on the number of tight contact structures we use convex surface theory for which we decompose our manifold in $4$ solid tori, toric annulus, and $2$ circle bundle over \{pair of pants\} as shown in Figure \ref{U1}. We start by maximizing the twisting numbers of the exceptional fibers by attaching bypasses. Then we use Honda's classification of tight contact structures on each of the pieces \cite{Hon02, Hon01}. We then glue these pieces together to construct $M$. It is possible to get an overtwisted contact structure when we glue two pieces with tight contact structures; we identify and discard such a combination. We look at the dividing curves on the convex surfaces we glue together to find overtwisted disks. Amongst the remaining combinations, we need to identify the isotopic ones from the others. We use relative Euler class \cite{Etnconvex} to identify non isotopic tight contact structures. \ 

\

 In Section \ref{2} we briefly look at the methods we use to get our lower and upper bound. Then we look at a couple of classification theorems that we use to classify tight contact structures on our Seifert fiberd manifolds. We give the proof of Theorem \ref{general thm}, in a special case in detail for better understanding and then in full generality in Section \ref{4}.\ 
 
\ 

It came to the author's notice 6 months before submission that Elif Medetogullari had similar results which were never published. 

\section{Preliminaries}
\label{2}
We expect the reader to be familiar with convex surface theory and contact surgery at the level of  \cite{Hon01} and \cite{Etn01}. We will only consider positively co-oriented contact structures, that is, we consider contact structures which satisfy the condition $\alpha\wedge d\alpha> 0.$ We recall a few theorems that we use in our proof here. 






\

A contact manifold $M$ is called $\textit{holomorphically fillable}$ if it is the oriented boundary of a compact Stein surface. It is proved by Eliashberg and Gromov that holomorphically fillable structures are tight \cite{EG}. 
\begin{thm}(Eliashberg)
\cite[Theorem 1.3]{Gom}
\label{HF}
	If $(M',\xi')$ is a contact manifold, obtained from a holomorphically fillable contact manifold ($M,\xi$) by Legendrian surgery then $(M',\xi')$ is holomorphically fillable.
\end{thm}
 It is proved by Eliashberg in \cite{Eli02} that $S^3$ with its standard contact structure is holomorphically fillable.  Using Theorem \ref{HF} we can construct tight contact structures by Legendrian surgery on knots in $S^3$ with standard contact structure. The following theorem determines when two such tight contact structures are isotopic.

\begin{thm}
\cite[Theorem 1.2]{LM01}
\label{LM} 
Let X be a smooth 4-manifold with a boundary. Suppose $J_1$ and $J_2$ are two Stein structures with boundary on X. Let $c(J)$ be the Chern class of the Stein structure $J$. If the induced contact structures on $\partial X$ are isotopic, then $c(J_1)=c(J_2)$.\
\end{thm}

Now we will state some classification results that we use to get the upper bound on the number of tight contact structures on Seifert fibered manifolds. One uses convex surface theory to obtain these results. We refer the reader to \cite{Gir02, Hon01, Hon02} for the basics of convex surface theory. Now we set notations as follows. Let $M$ be a 3-manifold with tight contact structures $\xi$. Let $\Sigma$ denote embedded (either closed or has Legendrian boundary) oriented surface in $M$. We denote the set of dividing curves on $\Sigma$ by $\Gamma_\Sigma$. The term $\textit{sign configuration}$ on a convex surface means the signs of the regions in the complement of dividng curves. Notice that when we say dividing set we mean the dividing curves and the sign configuration on $\Sigma$. Let $T^2$ be a minimal convex torus in standard form. The term minimal is used here to denote that the convex torus has $2$ dividing curves. Throughout the paper we will decompose and attach bypasses along minimal convex torus only and so we will omit the word minimal. We denote the slope of its dividing curves by $s$ or $s(T^2)$ when we want to specify the torus and we call it the slope of the torus. Say we have a 3-manifold $M$ with a convex torus as its boundary. By boundary slope, we refer to the slope of the dividing curves on the boundary torus and denote it by $s(\partial M)$. If our manifold has multiple boundary tori, we index them by $\partial(M)_i$ and denote the slopes by $s(\partial(M)_i)$. The slope of the Legendrian ruling is denoted by $r$. Let $\gamma\in M$ be a Legendrian curve. We denote the twisting number of a curve $\gamma$ by $t(\gamma)$. and the standard tubular neighborhood of $\gamma$ by  $N(\gamma)$. \ 

\ 

Here we state some results that we use to get the upper bound on the number of tight contact structures. Eliashberg gave the classification of tight contact structures on the 3-ball $B^3$.
\begin{thm} \cite[Theorem 2.1.3]{Eli03}
\label{B3}
Two tight contact structures on the ball $B^3$ which coincide at $\partial B^3$ are isotopic relative to $\partial B^3$.
\end{thm}
Now let us look at some classification results on $S^1\times D^2$ and on $T^2\times I$ which will be used in the classification of tight contact structures on Seifert fibered manifolds.



 First, we establish some notations.
Let $\frac{p_i}{q_i}\in (0,1)\cap \mathbb{Q}$ and $(p, q)=1$ we have the following unique continued fraction expansion:
$$
-\displaystyle\frac{q}{p}=a_{0}-\displaystyle\frac{1}{a_1-\displaystyle\frac{1}{a_{2}\cdots -\displaystyle\frac{1}{a_{k}}}},
$$
with all $a_{i}<-1$ integers. We identify $-\frac{q_i}{p_i}$ with $[a_0,\cdots,a_k]$.

\begin{thm} \cite[Theorem 2.3]{Hon01}\label{sd} 
Consider the tight contact structures on $S^1 \times D^2$ with convex boundary $T^2$, for which $\# \Gamma_{T^2}=2$ and $s\left(T^2\right)=-\frac{q}{p}$. There exist exactly $\left|\left(a_0+1\right)\left(a_1+1\right) \cdots\left(a_{k-1}+1\right)\left(a_k\right)\right|$ tight contact structures on $S^1 \times D^2$ with this boundary condition, up to isotopy fixing $T^2$.
\end{thm}
Now let us look at tight contact structures on $T^2\times I$.
We denote the slope of dividing curves on $T^2\times \{0\}$ by $s_0$ and the slope of dividing curves on $T^2\times \{1\}$ by $s_1$. We can assume that $s_0$ is $-1$. We write $s_1=-\frac{q}{p}$.\

\

The statement of classification of tight contact structures on $T^2\times I$ is in terms of the continued fraction expansion of the boundary slope.
\begin{thm} \cite[Theorem 2.2 (2a)]{Hon01} (Minimally twisting, rotative case)
\label{m}
 Let $\Gamma_{T_{i}}$, $i=0,1$, satisfy $\#\Gamma_{T_{i}}=2$ {\it and} $s_{0}=-1, s_{1}=-\frac{q}{p}$. Then
$|\pi_{0}(Tight^{\min}(T^{2}\times I,\ \Gamma_{T_{1}}\cup\Gamma_{T_{2}}))|\leq|(a_{0}+1)(a_{1}+1)\cdots(a_{k-1}+1)(a_{k})|$.

\end{thm}
\begin{thm}\cite[Theorem 2.2 (3)]{Hon01} (Minimally twisting, non-rotative case)
 Let $\Gamma_{T_i}$, i=0,1, satisfy $\#\Gamma_{T_i}=2$ and $s_0=s_1=-1.$ Then there exists a holonomy map $k:\pi_0(\Tight^{min}(T^2\times I, \Gamma_{T_1}\bigcup\Gamma_{T_2}))\rightarrow \mathbb{Z},$ which is bijective.
\end{thm}

\subsection{Giroux torsion}
\label{torsion}

Say we have a contact manifold $(M,\xi)$. Say $T\in M$ is a convex torus in standard form. We say that $(M,\xi)$ has $\textit{n-torsion}$ along $T$ if there exists a contact embedding of ($T^2\times I, \xi_n =\ker(\sin(2n\pi z)dx+\cos(2n\pi z)dy)$) into $(M,\xi)$, such that $T^2\times\{t\}$ are isotopic to $T$. We say that $(M,\xi)$ has $\textit{n-Giroux torsion}$ if there exists an embedded torus $T$ along which $(M,\xi)$ has $n$-torsion and there does not exist any embedded torus $T'$ along which $(M,\xi)$ has ($n+1$)-torsion \cite{GH}.

A 3-manifold $M$ is called $\textit{irreducible}$ if every 2-sphere $S^2\subset M$ bounds a ball $B^3 \subset M$. A 2-sided surface $S$ without $S^2$ or $D^2$ components is called $\textit{incompressible}$ if for each disk $D \subset M$ with $D\cap S=\partial D$ there is a disk $D'\subset S$ with $\partial D'=\partial D$. An irreducible manifold $M$ is called $\textit{atoroidal}$ if every incompressible torus in $M$ is $\partial$-parallel otherwise, the manifold is called $\textit{toroidal}$. It is proved by Honda, Kazez, and Mati\'c:

\begin{thm}
 \cite[Theorem 0.2]{HKM}
\label{GT1}
Let $M$ be an oriented, closed, connected, toroidal irreducible 3-manifold that contains an incompressible torus. Then $M$ carries $\mathbb{Z}$ many isotopy classes of tight contact structures.
\end{thm}
The converse was later proved by Colin, Giroux, and Honda:
\begin{thm}\cite[Theorem 0.2]{CGH}
\label{GT2}
Every closed, oriented, atoroidal 3-manifold carries a finite number of tight
contact structures up to isotopy.
\end{thm}
This tells us that the only source of $\mathbb{Z}$ many tight contact structures in a closed, connected, oriented 3-manifold is an incompressible torus. The proofs of Theorem \ref{GT1} and Theorem \ref{GT2} tell us that these $\mathbb{Z}$ many tight contact structures occur in a neighborhood, $T^2\times I$, of the incompressible torus. There are $\mathbb{Z}$ many non-isotopic tight contact structures on this $T^2\times I$ coming from Giroux $n$-torsion.\

\section{Tight structures on Seifert manifolds}\label{4}
In this Section, we will be giving the proof of Theorem \ref{general thm}, first in a special case and then in full generality. For details about the construction of Seifert fibered manifolds look at \cite{Hat}. We first look at a specific example of a Seifert fibered 3-manifold before proving it in full generality.

\subsection{Example case of Theorem \ref{general thm}}
\label{Example}
\begin{theorem*}[Special case of Theorem \ref{general thm}]    
There are three tight contact structures on $M(0,-4;1/2,1/2,1/2,\allowbreak 1/2)$  modulo Giroux torsion up to contact isotopy. All three of them are Stein fillable. For each $n\in\mathbb{Z}$ there exists at least one tight contact structure with n-Giroux torsion on $M$ up to contact isotopy. These tight contact structures are not weakly fillable.
\end{theorem*}

\begin{proof}
    
To construct this manifold we follow the method shown in \cite{Hat}. We start with the surface $S^2$. Let $X$ be $S^2$ with the interior of four disks removed and $M'$ be the circle bundle over $X$.

Identify each boundary component of $B'\times S^1$ with $\mathbb{R}^2/\mathbb{Z}^2$ by choosing $(1,0)^T$ to be the direction given by $-\partial(B'\times S^1)$ and $(0,1)^T$ to be the direction given by the $S^1-$fiber. Let $V_i's$ be the attaching solid tori. We identify
$\partial V_i$ with $\mathbb{R}^2/\mathbb{Z}^2$ by choosing $(1, 0)^T$ as the meridional direction and $(0, 1)^T$  as the longitudinal direction. The attaching maps $A_i:\partial V_i\rightarrow -\partial M'$ are given by $\left(\begin{array}{ll}
2 & -1\\
1 & \ \ 0
\end{array}\right)$.

The surgery representation of this manifold $M$ is shown on the left-hand side in Figure \ref{RT}. We perform four $(-1)$ Rolfsen twists to get the surgery diagram on the right in Figure \ref{RT}. The Legendrian representation of an unknot with surgery coefficient $-2$ has Thurston-Bennequin number $-1$ and hence the rotation number is $0$. This gives a unique Legendrian representation, for the four $-2$ framed unknots as shown in Figure \ref{R1}. For the unknot with surgery coefficient $-4$, the Legendrian realization has Thurston-Bennequin number $-3$, and hence the rotation number can be $-2,$ $0,$ or $2$. The corresponding Legendrian realizations are shown in Figure \ref{R2}. Since the rotation numbers of these three Legendrian realizations are different, the Chern numbers of the corresponding Stein structures are different \cite{LM}, hence, the three contact structures are non-isotopic. Since they are Stein fillable (using Lemma \ref{HF}) they have zero Giroux torsion \cite{Gay}. This tells us that there are at least three tight contact structures on $M$ with zero Giroux torsion up to contact isotopy. \ 
 
\begin{figure}
   \centering
    \includegraphics[width=2in]{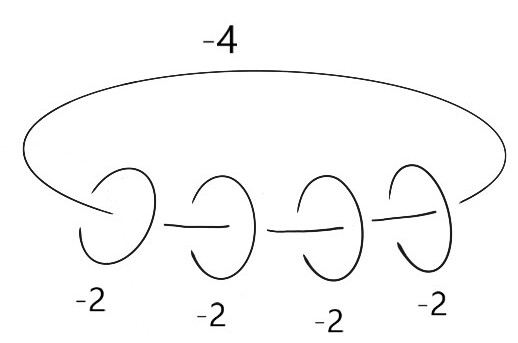}
     \caption{Surgery description of the Seifert manifold.}
    \label{RT}
\end{figure} 
 
\begin{figure}
   \centering
    \includegraphics[width=1in]{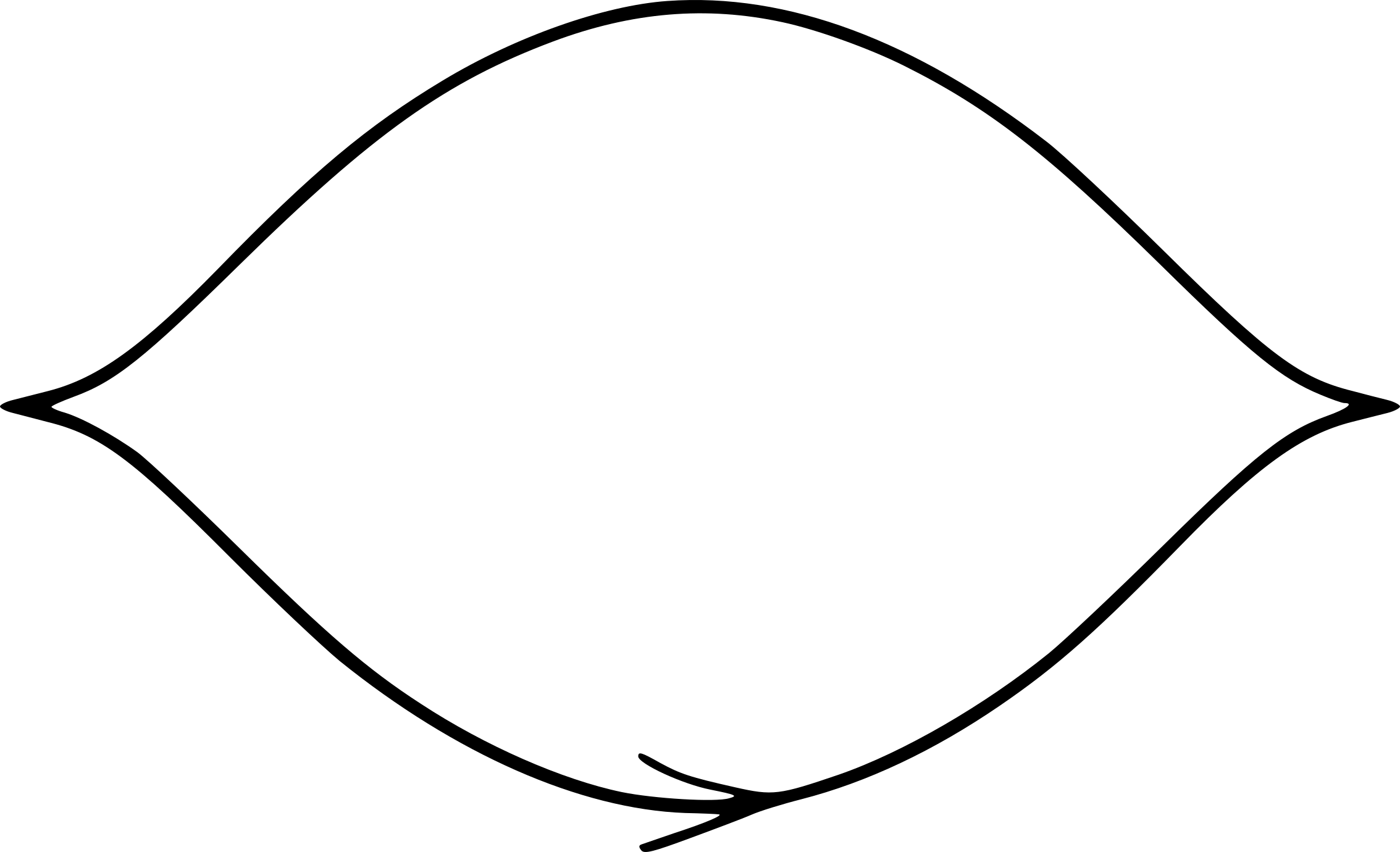}
     \caption{Legendrian realization of unknot with Thurston-Bennequin number -1.}
    \label{R1}
\end{figure} 

 \begin{figure}
   \centering
    \includegraphics[width=5in]{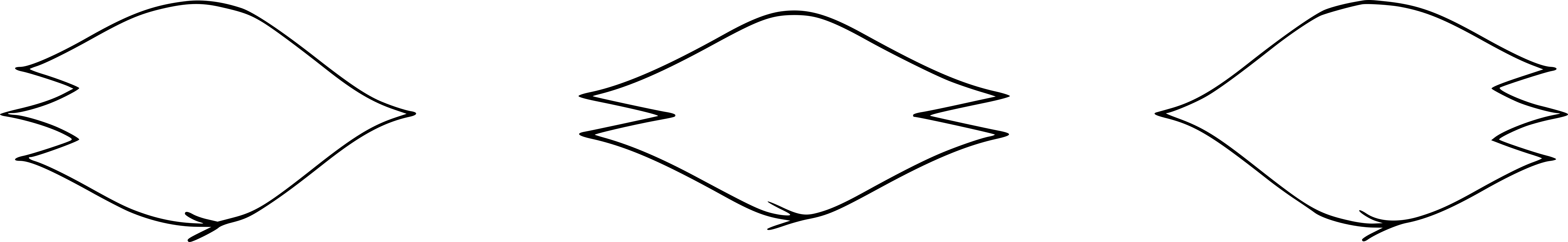}
     \caption{Legendrian realizations of the unknot with Thurston-Bennequin number -3.}
    \label{R2}
\end{figure}

\subsubsection{Upper bound}

\label{decompose}
The Seifert fibered 3-manifold with four exceptional fibers, $M(0,-4;1/2,1/2,1/2,\allowbreak1/2)$ has an incompressible torus. We would like to count the tight contact structures modulo this $\mathbb{Z}$ many tight contact structures. Consider a $T^2\times I$ neighborhood of one of the incompressible tori. We cut along $T_0$ and $T_1$ to decompose our manifold in three pieces, one of which is $T^2\times I$. The other two pieces are denoted by  $N_1$ and $N_2$. As shown in Figure \ref{U1} each $N_i$ can be decomposed in \{pair of pants\}$\times S^1$ and two solid tori.

 \begin{figure}
   \centering
    \includegraphics[width=3in]{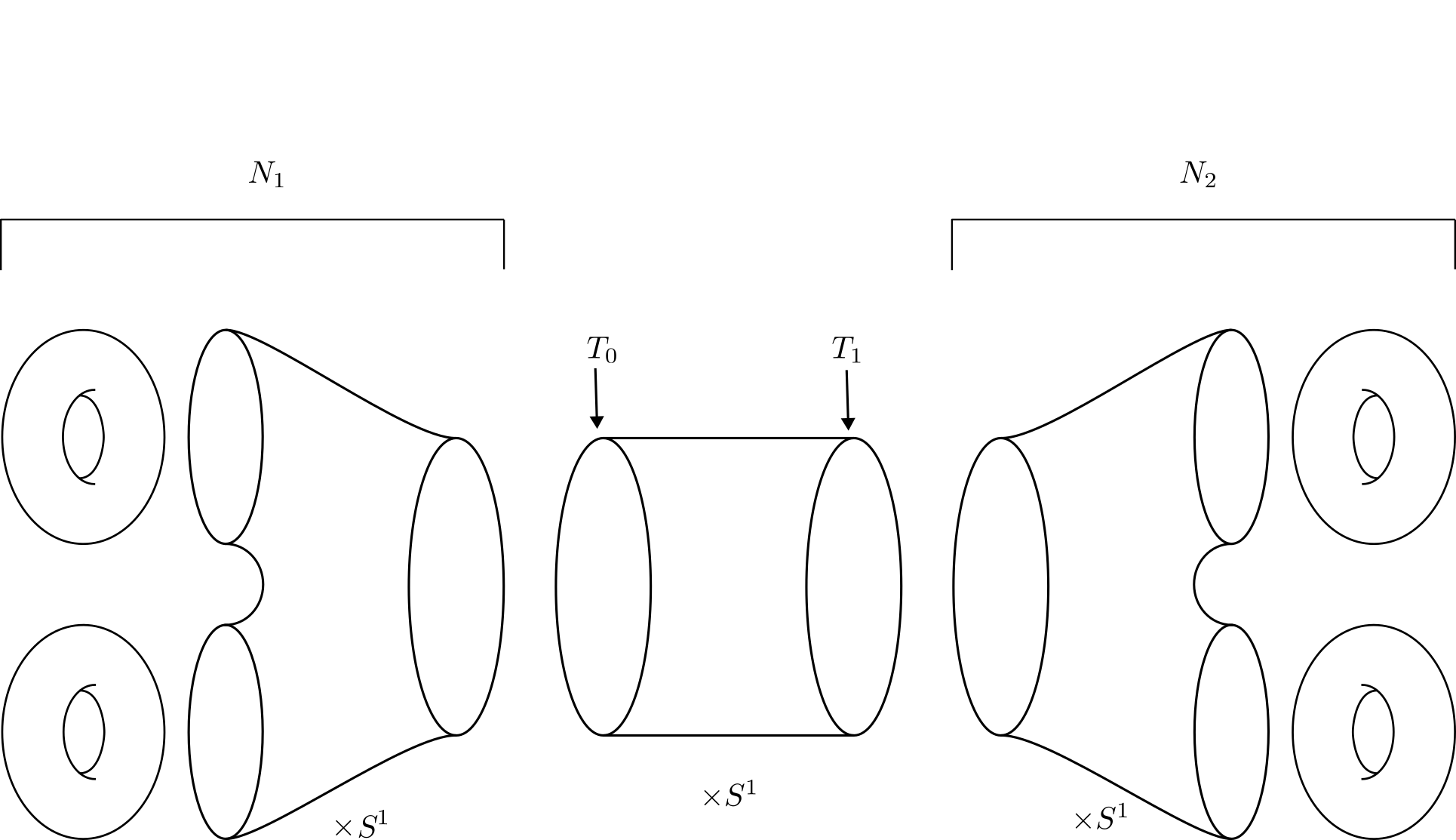}
     \caption{Decomposition of a Seifert fibered 3-manifold with four exceptional fibers.}
    \label{U1}
\end{figure}

\ 

Let us start by looking at one of these pieces, say $N_1$. We denote the \{pair of pants\}$\times S^1$ by $\Sigma\times S^1$ and the two solid tori by $V_1, V_2$. There are three boundary components of $\Sigma\times S^1$, which we denote by $\partial(\Sigma\times S^1)_i$, where $\partial(\Sigma\times S^1)_{i}$ for $i=1,2$ corresponds to the torus $\partial V_{i}$ for $i=1,2$ and $\partial(\Sigma\times S^1)_{3}$ is glued to the $T^2\times I$. Note that $\partial(\Sigma\times S^1)_{3}$ is the same as $\partial N_1$.  We denote the two singular fibers by $F_1$ and $F_2$ and their twisting numbers by $n_1$ and $n_2$. Since $V_1$ and $V_2$ are standard neighborhoods of $F_1$ and $F_2$. Assume that these are simultaneously isotoped to Legendrian curves and further isotoped so that their twisting number is negative. The slope of the dividing curves (Example 1.4.12 in \cite{Wan}) on $\partial V_i$ is $\displaystyle\frac{1}{n_i}$. Using the construction from \cite{Hat}, the two attaching maps $A_i:\partial V_i\rightarrow -\partial (\Sigma\times S^1)_i$ are given by $\left(\begin{array}{ll}
2 & -1\\
1 & \ \ 0
\end{array}\right)$ for $i=1,2$. 
Note that the slope of dividing curves on $\partial(\Sigma\times S^1)$ is not $\infty$. Using the flexibility of Legendrian ruling (Corollary 3.6 in \cite{Hon01}) we may assume that the Legendrian ruling slope of $\partial(\Sigma\times S^1)_{i}$ for $i=1,2,3$ is infinite. 
\ 
\subsubsection{Maximising twisting numbers}
We use the same methods to maximize the twisting number as illustrated in \cite{GS}. For $i=1,2$ we have $A_i.(n_i,1)^T=(2n_i-1,n_i)$. We denote the slope of the dividing curve on $-\partial(\Sigma\times S^1)_i$ by $s(\partial(\Sigma\times S^1)_i)=\displaystyle\frac{n_i}{2n_i-1}$. \ 

\ 

\begin{lem}\label{mtn}
We can increase the twisting numbers $n_1$ and $n_2$ up to $0$.
\end{lem}

\begin{proof}
Consider an annulus $I\times S^1$ in $M$ from $\partial(\Sigma\times S^1)_1$ to $\partial(\Sigma\times S^1)_2$ such that its boundary consists of Legendrian ruling curves on the tori. The boundary of this annulus intersects the dividing curves in $2(2n_i-1)$ points respectively. \ 

\ 

If $n_1\neq n_2$ then, due to the imbalance principle (\cite{Hon}) there exists a bypass along a Legendrian ruling curve on either of the boundaries. Note that $A_i^{-1}=\left(\begin{array}{ll}
0 & 1\\
-1 & 2
\end{array}\right)$ and hence the Legendrian ruling has slope $2$ on $\partial V_i.$ Using the twist number Lemma we can attach a bypass and thicken $V_i$ to $V_i'$ increasing the twisting number as long as $n_i< 0.$ We can apply the imbalance principle till $n_1= n_2$. Since $n_1=n_2$ there are an equal number of bypasses on both ends. We attach any remaining bypasses, if any, on either of the boundaries. Once you attach all the bypasses, there will be no boundary parallel dividing curves and, hence all the dividing curves will run from one boundary component to the other.  \ 
 
\ 

Now, we can assume that $n_1=n_2$ and there is no bypass on the annulus $A=I\times S^1$ in $M$ from $\partial(\Sigma\times S^1)_1$ to $\partial(\Sigma\times S^1)_2$. Then we cut $N_1\setminus(V_1'\cup V_2')$ open along $A$. Note that a neighborhood of $A\cup V_1'\cup V_2'$ is a piecewise smooth solid torus with four edges. We use the edge rounding Lemma (Lemma 3.11, \cite{Hon01}) to smoothen these four edges. Since each rounding changes the slope by an amount of $-\frac{1}{4}\frac{1}{2n_1-1}$, the slope of the diving curves on the boundary torus is
$$
\label{cnr}
s(\partial(\Sigma\times S^1)_1) +s(\partial(\Sigma\times S^1)_2)-4(\frac{1}{4}\frac{1}{2n_1-1})
$$
$$
\frac{n_1}{2n_1-1}+\frac{n_1}{2n_1-1}-\frac{1}{2n_1-1}=1.
$$
This boundary torus is isotopic to $\partial N_1$ and identified with $\mathbb{R}^2/\mathbb{Z}^2$ in the same way as $\partial N_1$. Hence the slope of the dividing curves on boundary torus $-\partial N_1$ is $-1.$\ 

\ 

Now take an annulus $I\times S^1$ from $\partial(\Sigma\times S^1)_1$ to $\partial N_1$. For $n_1<0$ we have $2n_1-1<-1$. Hence there exists a bypass by the imbalance principle on $V_1$ until we increase $n_1$ up to $0$. One can do a similar calculation for $n_2.$ Hence we can increase $n_1$ and $n_2$ until $n_1=n_2=0.$\ 

\end{proof}
\ 

\subsubsection{Combining tight contact structures on basic blocks}
We have $n_1=n_2=0$. The slope of the dividing curves on $-\partial N_1$ is $-1$, on $\partial(\Sigma\times S^1)_1
$ is $0$ and on $\partial(\Sigma\times S^1)_2$ is $0.$ Now we count the number of tight contact structures on $\Sigma \times S^1$ when the slope of the dividing curves on the three boundary tori is $0,0,-1$. Again consider an annulus $I\times S^1$ from $\partial(\Sigma\times S^1)_1$ to $\partial(\Sigma\times S^1)_2$. Either there exists a bypass on both boundary components or not. \ 

\ 

Case 1: If no bypass exists then we have the following conditions: 
According to the classification Lemma 5.1 of Honda \cite{Hon02}, $\Sigma\times S^1$ with such boundary slopes has a unique tight contact structure as shown in Figure \ref{A} (Note that we are using the opposite sign convention to Honda's). Call it $\xi_A$. Using the $A_i^{-1}$ the slope of the dividing curve on the boundary of $V_{i}$ for $i=1,2$ is $\infty.$ By Proposition \ref{sd}, there is exactly one tight contact structure on $V_{i}$ for $i=1,2$. This gives a unique tight contact structure on $N_1.$\ 

\begin{figure}
   \centering
    \includegraphics[width=1.5in]{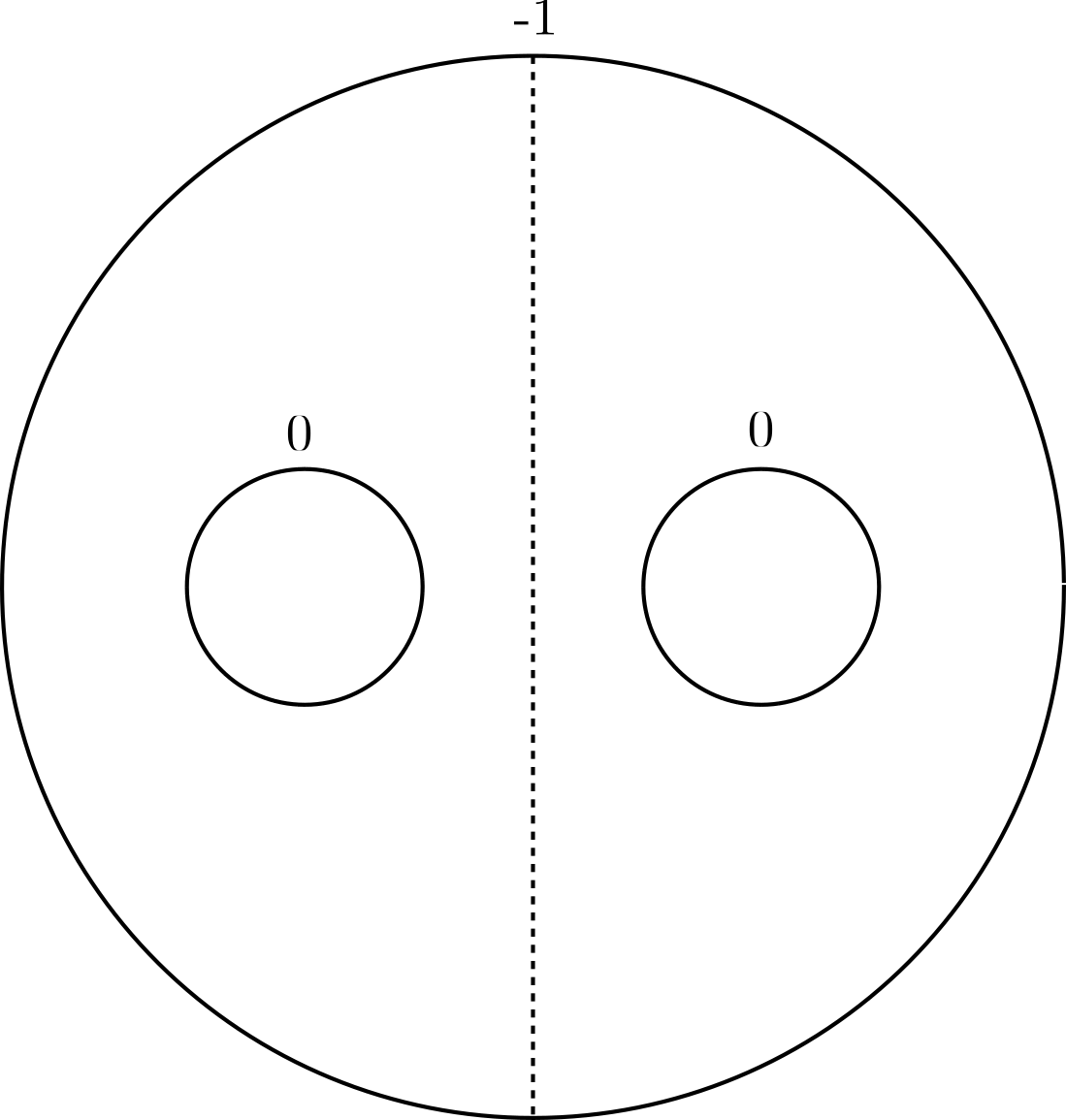}
     \caption{Dividing curve on $\Sigma$ with contact structure $\xi_A$ on $\Sigma\times S^1$.}
    \label{A}
\end{figure}\

\ 

Case 2: If there is a bypass then the cutting and rounding construction (see edge rounding in \cite{Hon}) gives a torus of infinite slope after a bypass attachment. Then Honda's classification of tight contact structures on $\Sigma \times S^1$ (Lemma 5.1, \cite{Hon02})  asserts that there exists a unique factorisation $\Sigma\times S^1= \Sigma'\times S^1 \cup L_1\cup L_2\cup L_3$, where the $L_i$ are $T^2\times I$ with minimal twisting and all components of the boundary of $\Sigma'\times S^1$, denoted by $\partial(\Sigma'\times S^1)_{i}$, have dividing curves of $\infty$ slope. Figure \ref{W1} shows $\Sigma$ and $\Sigma'$ with their boundary slopes. Here we fix $\Sigma$ to be $\Sigma \times \{0\}$ and $\Sigma'$ to be $\Sigma'\times \{0\}.$

\begin{figure}
   \centering
     \includegraphics[width=2in]{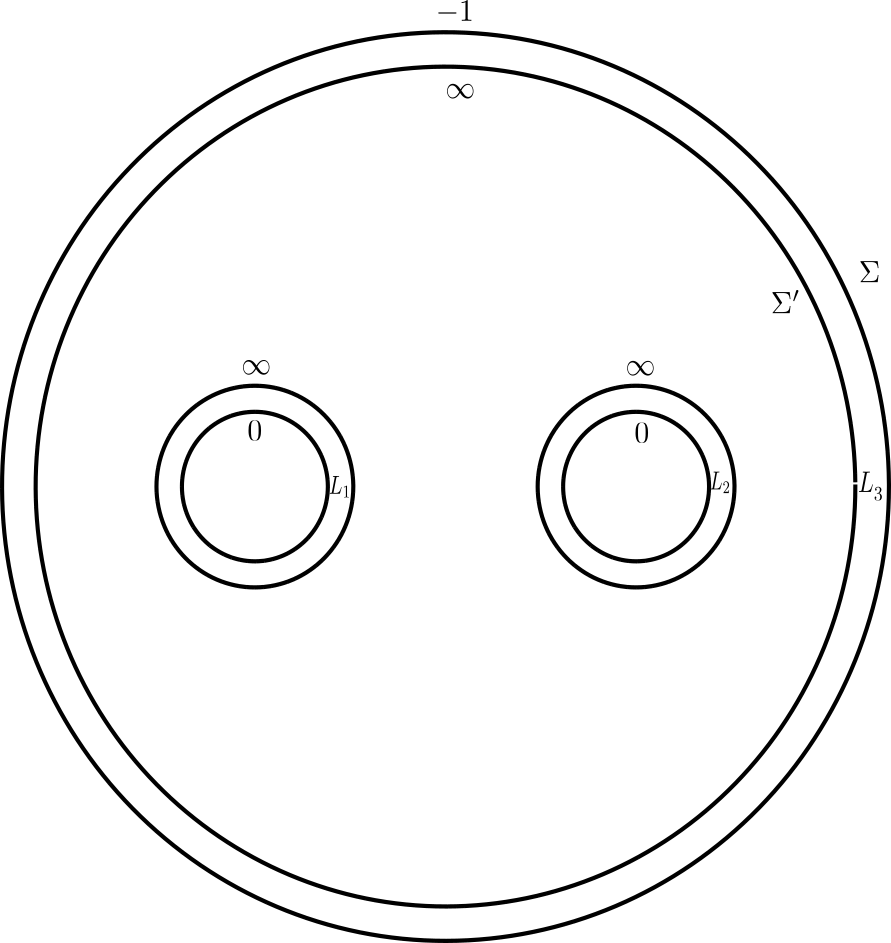}
     \caption{$\Sigma$ and $\Sigma'$ with their boundary slopes.}
    \label{W1}
\end{figure}\

Let us start by looking at the tight contact structures on $\Sigma'\times S^1$ and then add in the $L_i$ to get the tight contact structures on $\Sigma\times S^1$. Each boundary component of $\Sigma'$ intersects the dividing set of the corresponding torus twice.
\begin{figure}
   \centering
     \includegraphics[width=3in]{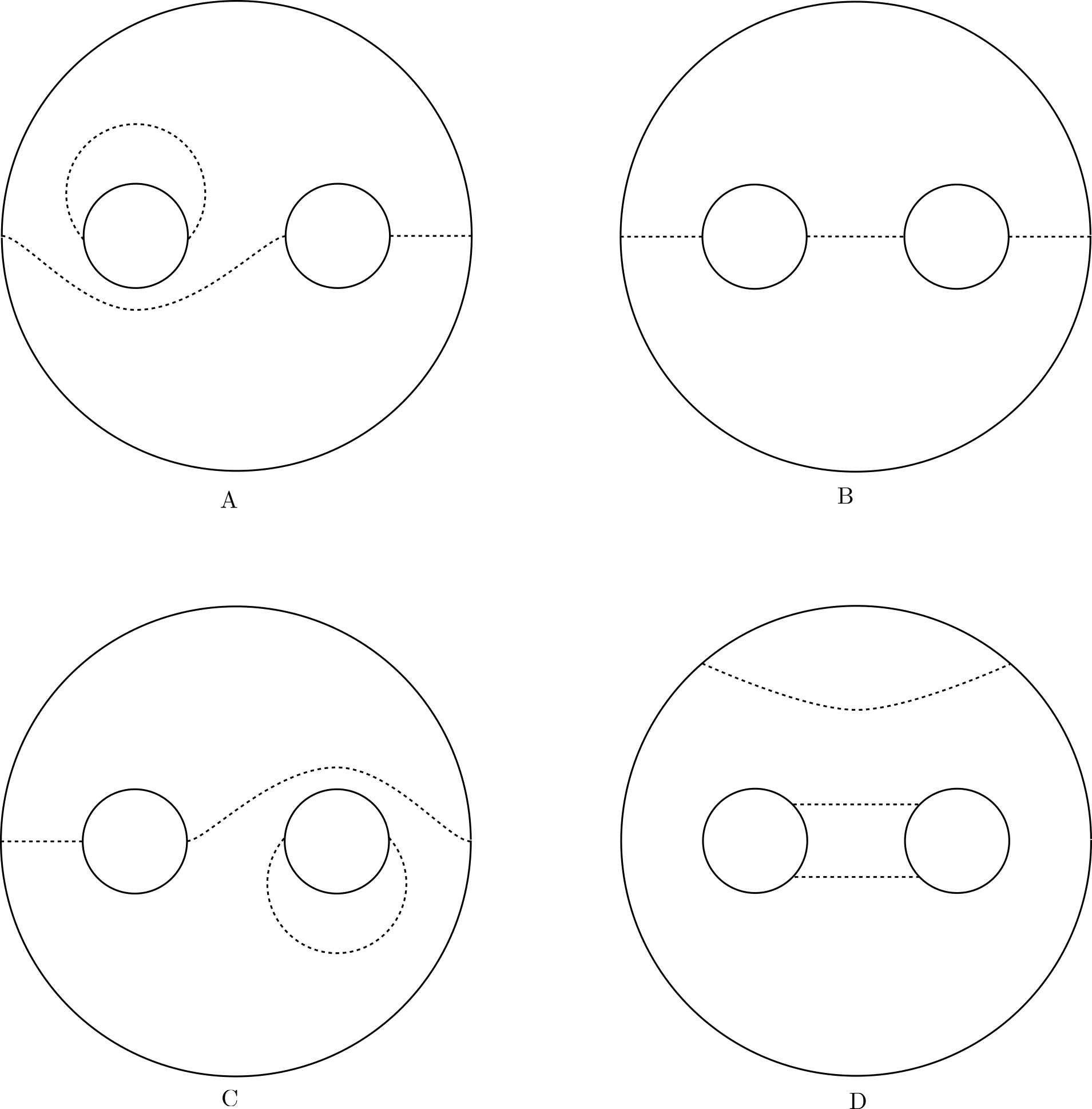}
     \caption{Possible dividing curves on pair of pants.}
    \label{pop}
\end{figure}\ 
\begin{lem}
\label{dvc}
Dividing curves on $\Sigma'$ either
connect boundary components in pairs as in Figure \ref{pop} B, or we have one boundary parallel dividing curve on $\partial(\Sigma\times S^1)_{3}$ and two dividing curves connecting $\partial(\Sigma\times S^1)_{1}$ to $\partial(\Sigma\times S^1)_{2}$ as shown in Figure \ref{pop} D. 
\end{lem}
\begin{proof}
We call $\partial(\Sigma\times S^1)_{1}$ as boundary component one. Similarly for boundary component two. Assume there is a boundary-parallel dividing arc on boundary component one or on boundary component two as shown in Figure \ref{pop} A, C. Say the boundary parallel dividing arc is across  $\partial(\Sigma\times S^1)_{1}$. This means there is a bypass along $\partial(\Sigma\times S^1)_{1}$. After attaching this bypass we can thicken $V_1\cup L_1$ to get $V_1'$. The slope on the boundary of $V_1'$ is $0$ in the basis of $\Sigma'\times S^1$. Hence we have a toric annulus $L_1$ with boundary slopes $0$ and $\infty$ and an extension of this toric annulus is another toric annulus with boundary slopes $\infty$ and $0$. The slope goes from $0$ to $\infty$ to $0$ on $V_1'$ and hence the contact structures on $V_1'$ are overtwisted. Similarly, we get an overtwisted contact structure if we have a boundary-parallel dividing arc on boundary component two. The possible dividing curve configurations without boundary-parallel dividing arcs on boundary component one or on boundary component two are of the form described in this Lemma.


\end{proof}\ 
From the result in Lemma \ref{dvc} we can divide our analysis into two cases, corresponding to Figure \ref{pop} B or D. These are two different dividing sets corresponding to different contact structures on $\Sigma\times S^1$. We will first count all the tight contact structures we get corresponding to Figure \ref{pop} B. We refer to this analysis as Case 2A. Then we do the same for Figure \ref{pop} D and refer to it by Case 2B.
 
\ 

Case 2A: Consider the set of dividing curves where each curve on $\Sigma'$ connects one boundary component to the other as shown in Figure \ref{pop} B. The following Lemma is proved by Honda and Etnyre \cite{EH}. I restate it for clarity. 
\begin{lem}
\label{unique}
The two dividing sets given by two different sign configurations in Figure \ref{pop} B give a unique contact structure on $\Sigma'\times S^1.$ 
\end{lem}
\begin{proof}
We start by cutting $\Sigma'\times S^1$ along $\Sigma'$ and then round the edges (see edge rounding in \cite{Hon}). We get a solid genus-two handlebody. We can arrange the dividing curves on the boundary so that two meridional disks intersect the dividing set exactly twice. (This is shown in Figure \ref{q1}.) Hence there is a unique dividing curve, separating the two intersection points, on these two disks. We cut along these two disks to get a 3-ball. There is a unique contact structure on this 3-ball with the given restriction to the boundary surface (Theorem \ref{B3}). Since the dividing curves on the surface, we cut along are determined by our initial configuration of dividing curves, we get a unique contact structure on $\Sigma'\times S^1$.
\begin{figure}
   \centering
    \centering
    \includegraphics[width=2in]{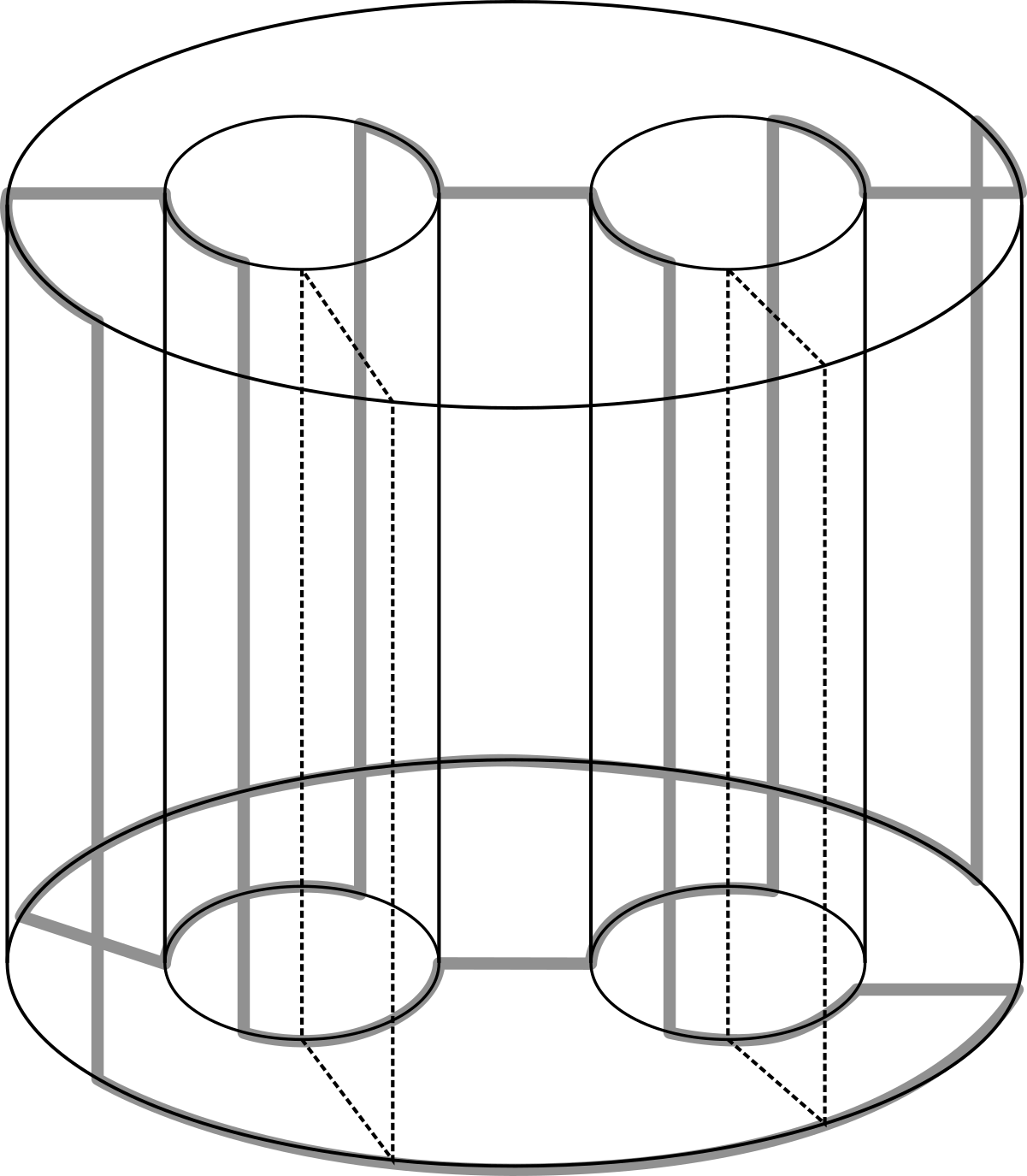}
    
     \caption{Dividing curves on boundary of $\Sigma'\times I$.}
    \label{q1}
\end{figure} 
\end{proof}\ 

\ 

\begin{figure}
   \centering
    \includegraphics[width=2.5in]{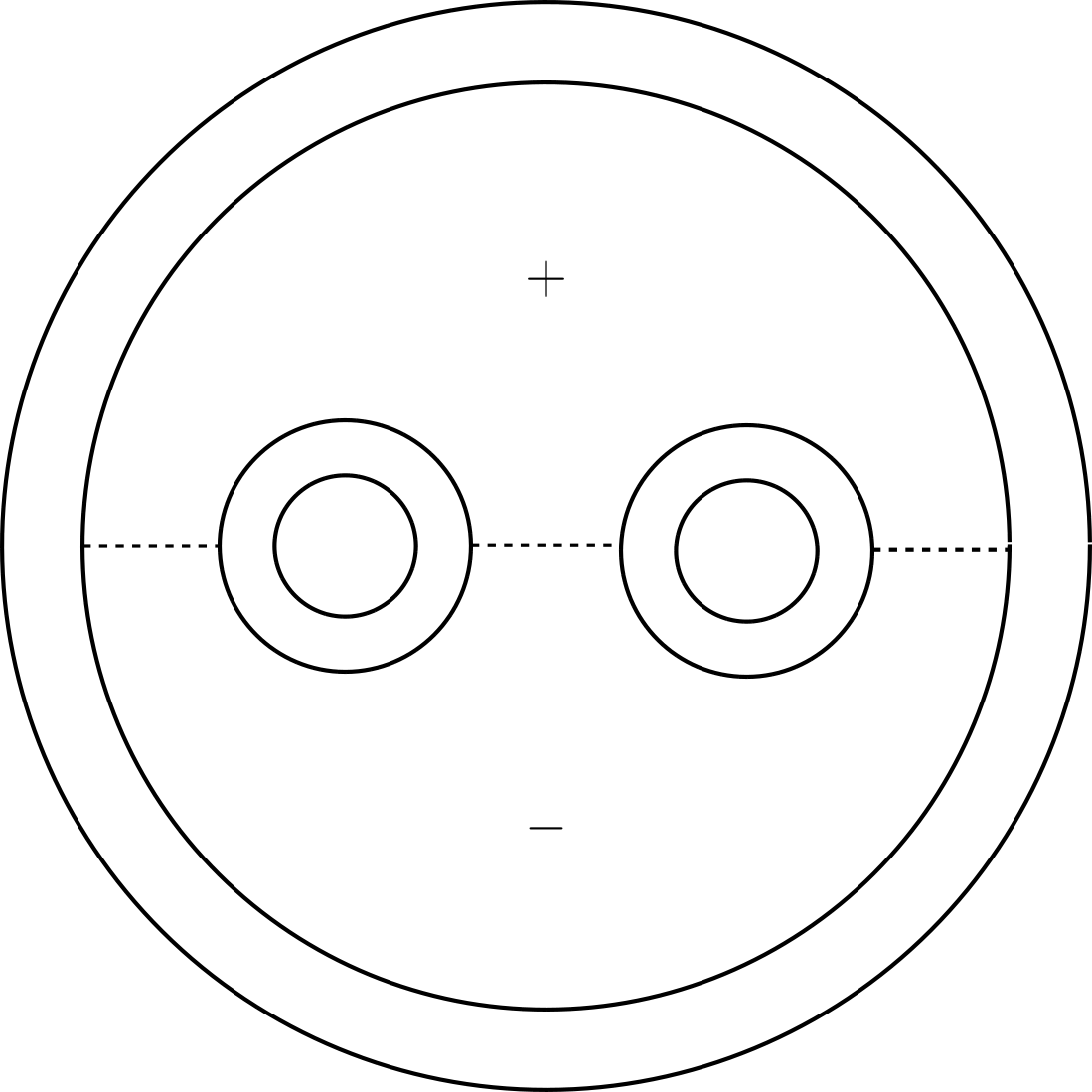}
     \caption{Possible dividing set on $\Sigma$.}
    \label{Z1}
\end{figure}\

Case 2B: Let us look at the case when one dividing curve goes from $\partial(\Sigma\times S^1)_{3}$ to itself and two dividing curves go from $\partial(\Sigma\times S^1)_{1}$ to $\partial(\Sigma\times S^1)_{2}$ as shown in Figure \ref{pop} D. 
This gives us at most two tight contact structures on $\Sigma \times S^1$ for Case 2B, one for each sign configuration.

\begin{figure}
   \centering
    \includegraphics[width=3in]{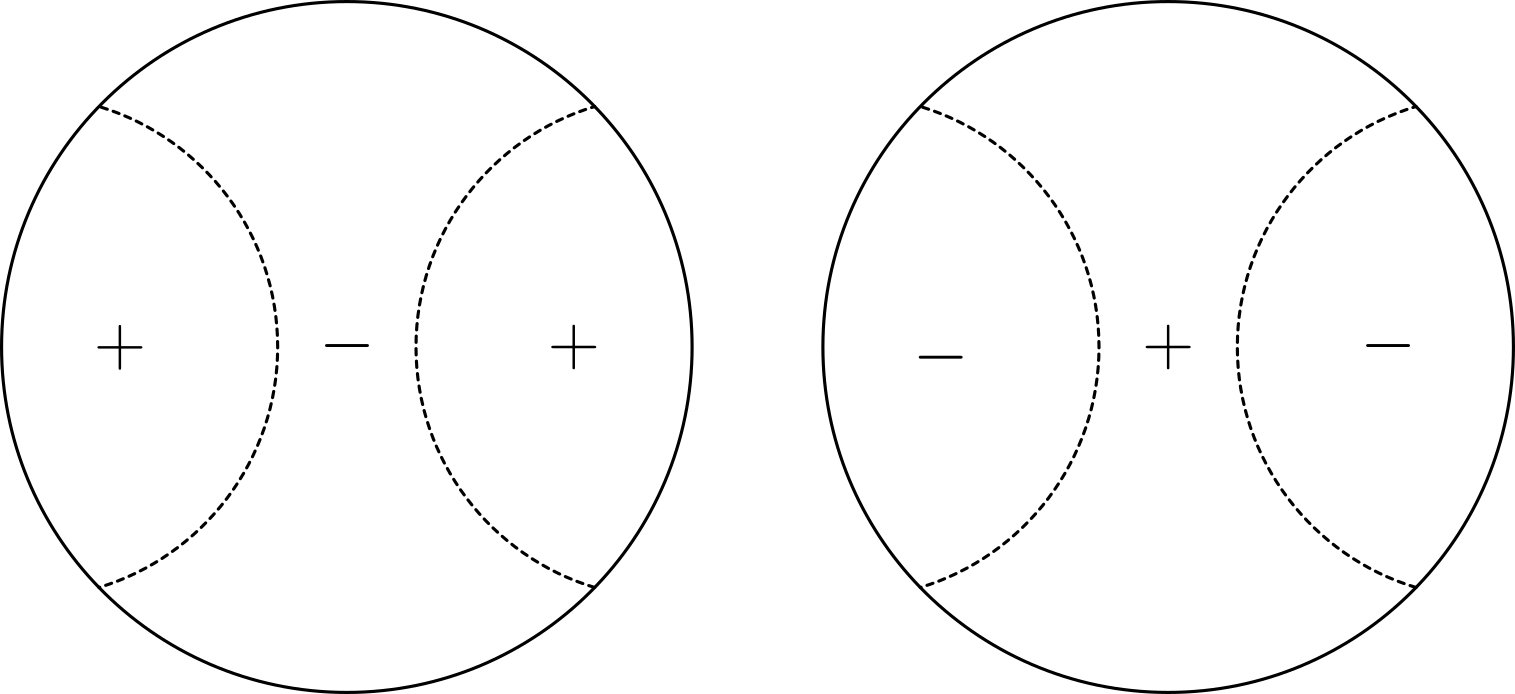}
     \caption{Two possibilities of dividing curves on $\mathbb{D}^2$ with t$(\partial \mathbb{D})= \minus 2$ \cite{Hon01}.}
    \label{d2}
\end{figure}

\begin{figure}
   \centering
    \includegraphics[width=2in]{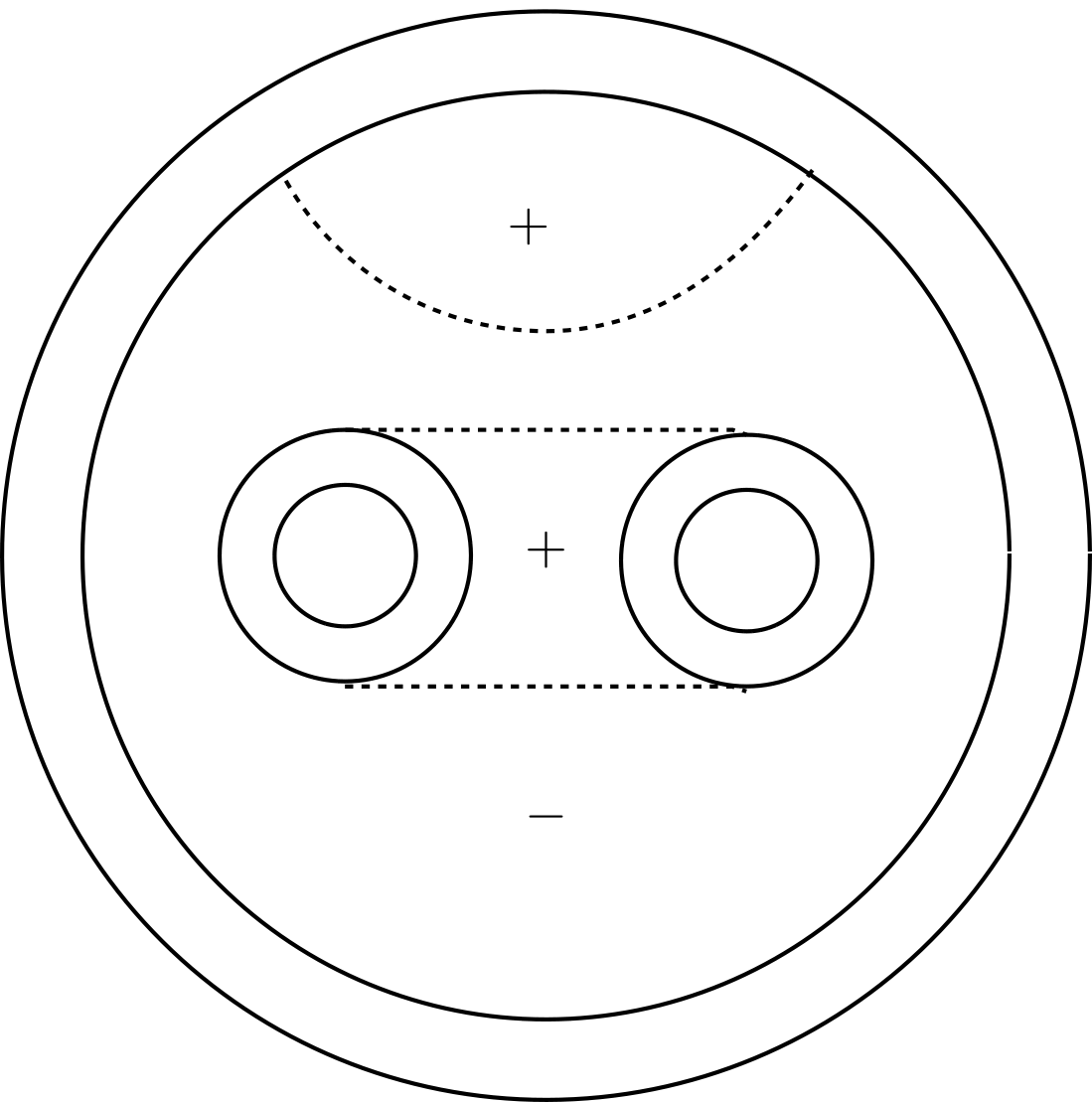}
     \caption{Possible dividing set on $\Sigma$.}
    \label{Z2}
\end{figure}\ 

\ 

We now look at the dividing curves on $\Sigma\times S^1=\Sigma'\times S^1\cup L_1\cup L_2\cup L_3.$ Let $A_i=\Sigma\cap L_i$. For the contact structure on $\Sigma\times S^1$ to be tight the dividing set of $A_3$ will have two arcs connecting the boundary components. There are two possible configurations depending on the sign of the basic slice $L_3$. 
The dividing set of $A_1$ and $A_2$ consists of a boundary parallel arc on the boundary component with dividing curves of the infinite slope. There are two possible configurations depending on the sign of the respective basic slice. This is shown in Figure \ref{s2}. Again dotted lines represent the dividing curves. Now let us look at the dividing curves on $\Sigma\times S^1$ separately for Case 2A and Case 2B as we did before.
\begin{figure}
   \centering
    \includegraphics[width=3in]{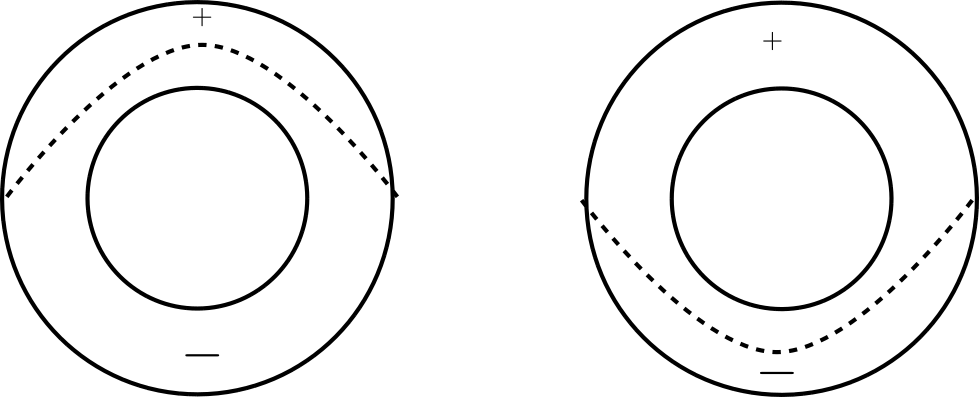}
     \caption{Positive (left) and Negative (right) sign configuration of dividing sets on $A_2$ and $A_1$.}
    \label{s2}
\end{figure}\ 

\

Case 2A: We fix one sign configuration for the contact structure on $\Sigma'\times S^1$ as shown in Figure \ref{Z1}. This fixes the sign configuration of the contact structure on $L_3$. In our case it is positive. We write $(\pm,\pm,\pm)$ to denote that the signs of the basic slices, $L_i$ for $i=1,2,3$ are positive/negative. Now we have a choice of sign for the contact structure on $L_1$ and $L_2$. This gives us four total configurations, given by: $(+,+,+), (-,-,+), (-,+,+), (+,-,+)$. It is proved in \cite{Hon02} (Lemma 5.1, 9th paragraph in the proof) that if all three basic slices have the same sign then we get an overtwisted disk. Hence $(+,+,+)$ corresponds to an overtwisted contact structure on $\Sigma'\times S^1$.\ 

\ 

We can have a contact structure with $(-,-,+)$ configuration. This corresponds to one tight contact structure on $\Sigma\times S^1$ with the slopes of the dividing curve on the boundary torus $0,0,-1.$ The dividing set is shown in Figure \ref{B1}. The tight contact structure on $\Sigma\times S^1$ corresponding to this dividing set on $\Sigma$ will be denoted by $\xi_B$. If we had started our Case 2A with the opposite signed configuration, then we would have the dividing set as shown in Figure \ref{B'} on $\Sigma$. This corresponds to a different tight contact structure because the relative Euler classes are different on $\Sigma\times S^1$. Let us call it $\xi_{B'}$.

\begin{figure}
   \centering
    \includegraphics[width=2in]{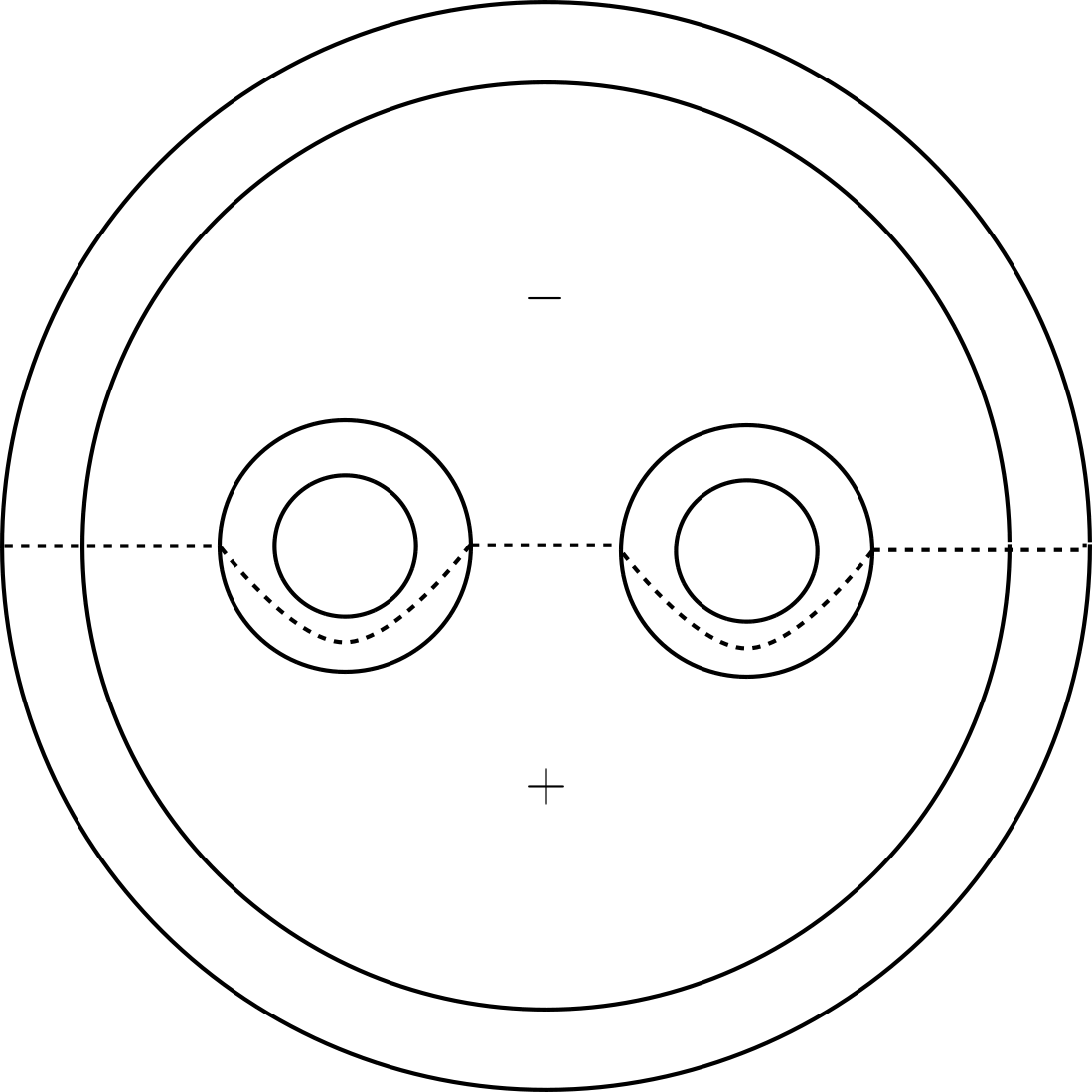}
     \caption{Dividing set of $\xi_B$ on $\Sigma$.}
    \label{B1}
\end{figure}

\begin{figure}
   \centering
    \includegraphics[width=2in]{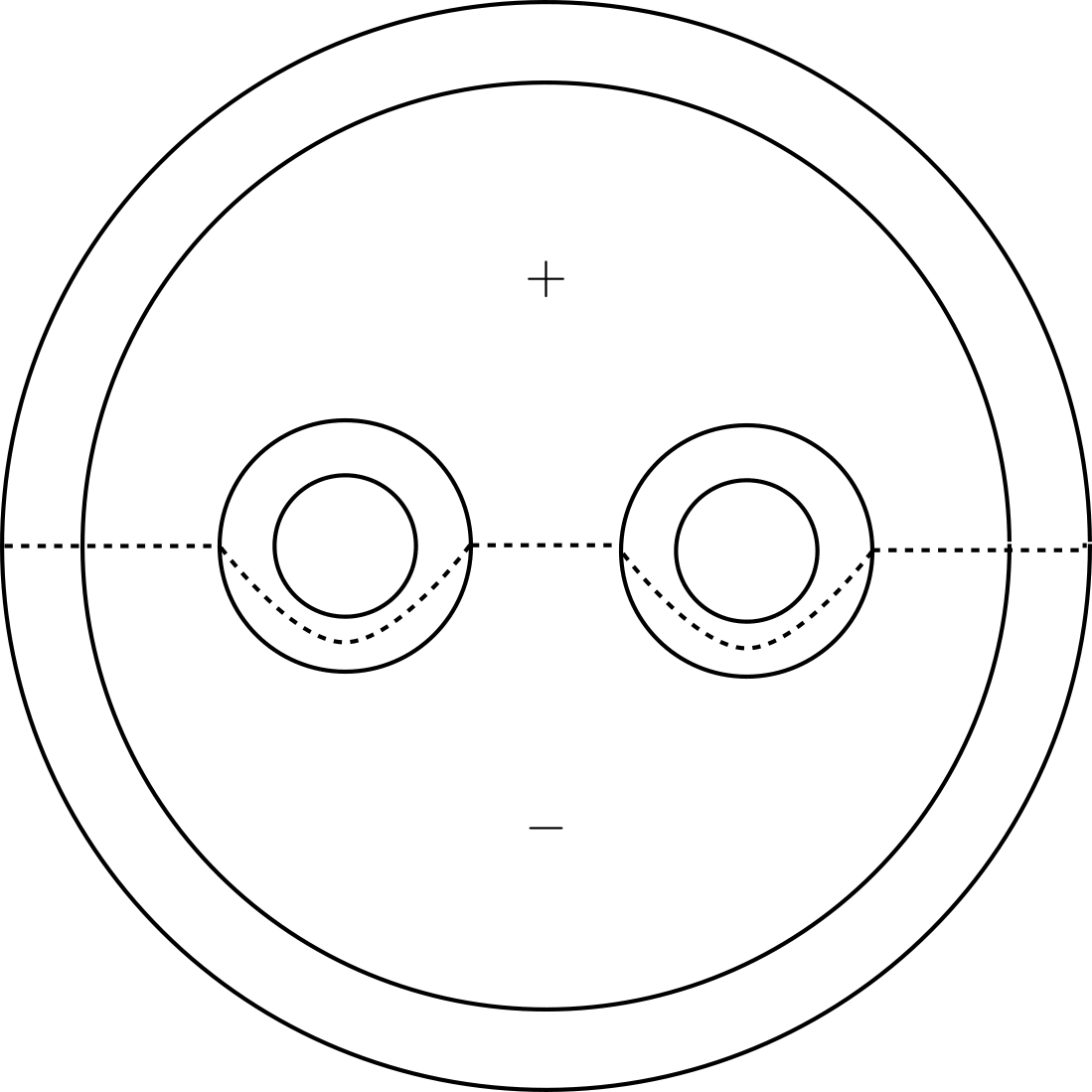}
     \caption{Dividing set of $\xi_{B'}$ on $\Sigma$.}
    \label{B'}
\end{figure}\ 

\ 

Similarly we can have $(-,+,+)$ and $(+,-,+)$ configuration. The dividing sets corresponding to these signed configurations are shown in Figure \ref{C, D}. Each of these corresponds to one tight contact structure on $\Sigma\times S^1$. We denote them by $\xi_C$ and $\xi_D$. If we had started our Case 2A with the opposite signed configuration, then we would have the dividing sets shown in Figure \ref{C', D'} on $\Sigma$. These correspond to different tight contact structures on $\Sigma\times S^1$. Let us call them $\xi_{C'}$ and $\xi_{D'}$. \ 

\ 

\begin{figure}
   \centering
    \includegraphics[width=3in]{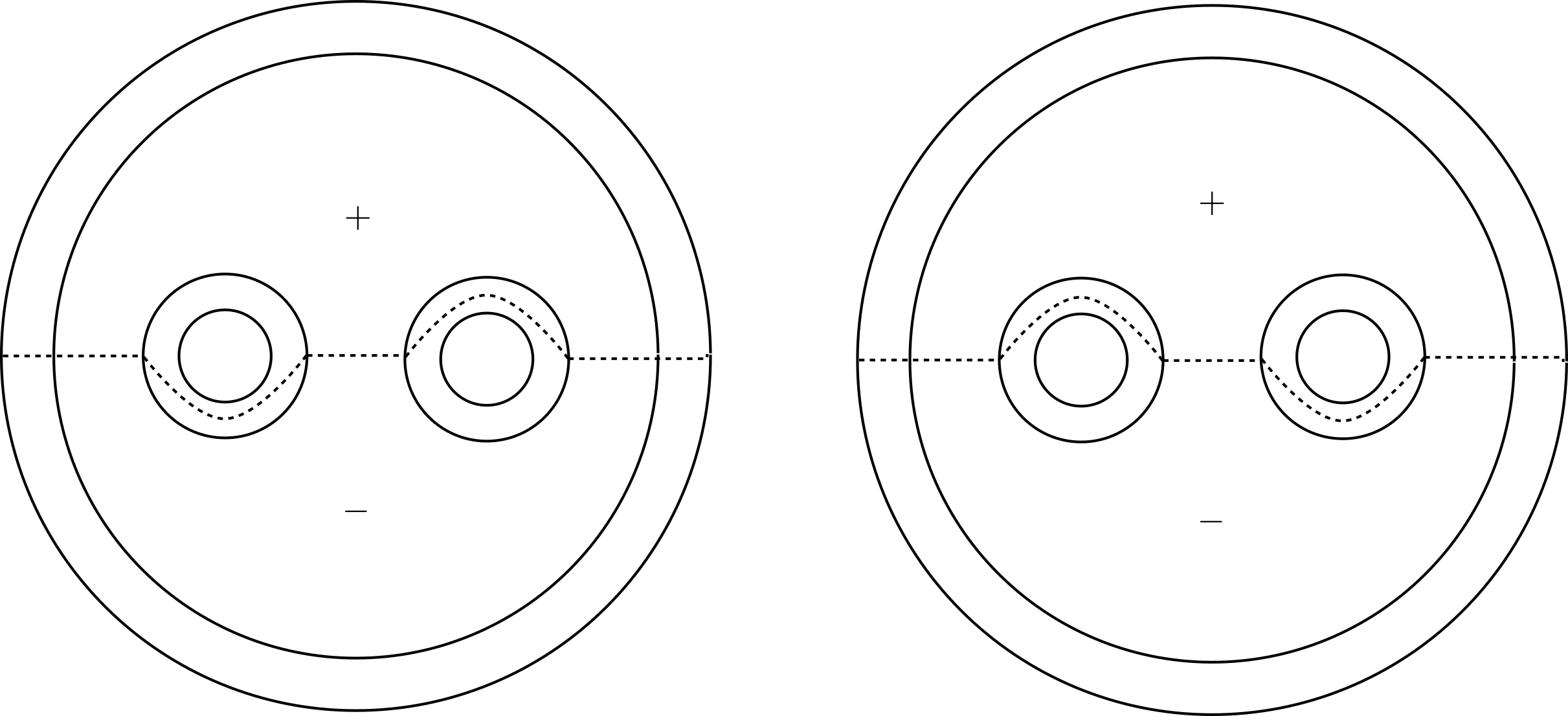}
     \caption{Dividing set of $\xi_C$ and $\xi_D$ on $\Sigma$.}
    \label{C, D}
\end{figure}

\begin{figure}
   \centering
    \includegraphics[width=3in]{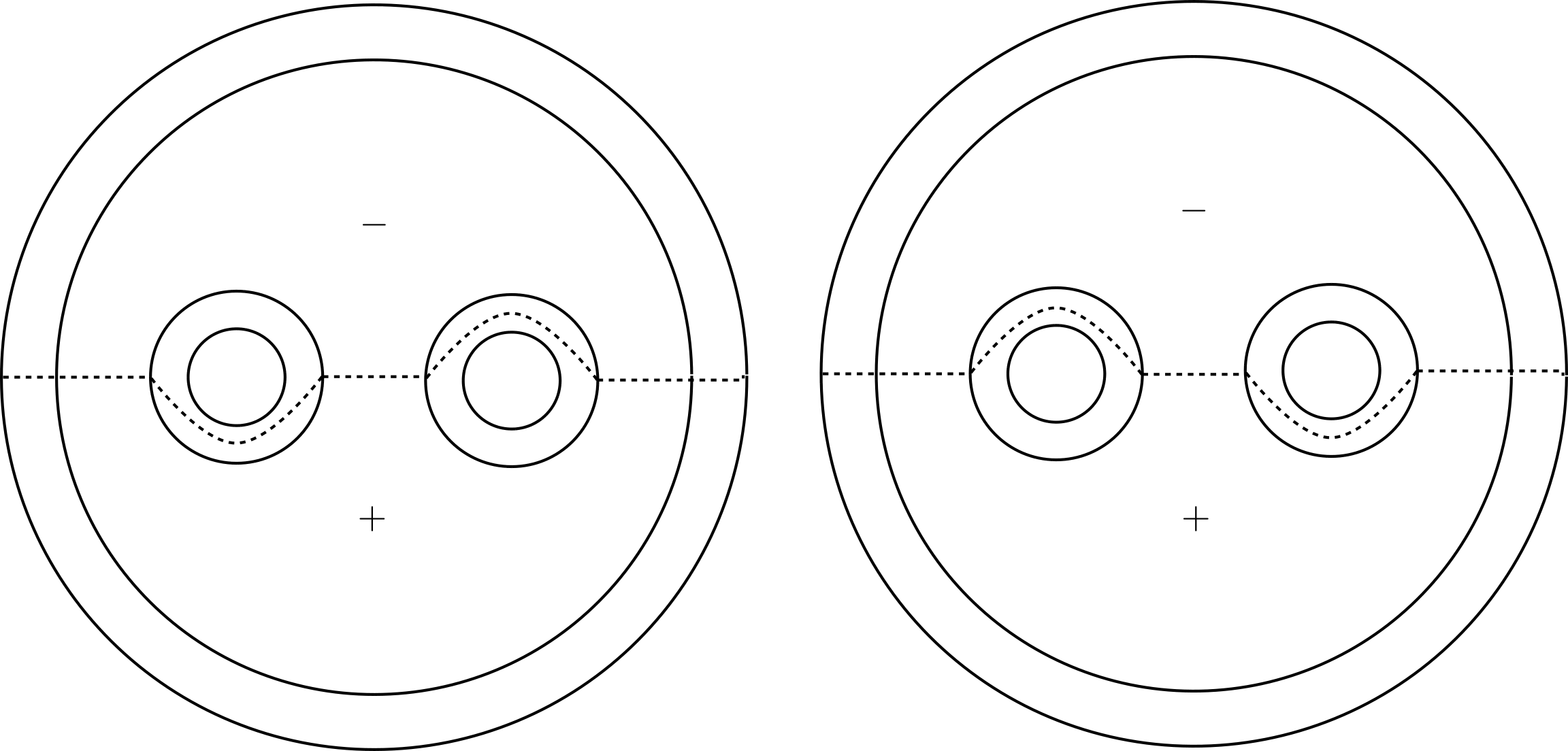}
     \caption{Dividing set of $\xi_{C'}$ and $\xi_{D'}$ on $\Sigma$.}
    \label{C', D'}
\end{figure}\

Case 2B: We have fixed one sign configuration for $\Sigma'\times S^1$ as shown in Figure \ref{Z2}. This fixes the sign configuration of $L_3$. Now we have a choice for sign-on $L_1$ and $L_2$. This gives us four total configurations shown in Figure \ref{dc}. The case where both the $L_1$ and $L_2$ have positive signs yields an overtwisted contact structure. We can see the dividing curves by dotted lines in Figure \ref{dc} A. The boundary of the overtwisted disk is outside the disk-bound by these dividing curves. Now consider the case when $L_1$ and $L_2$ have mixed signs. Then we get a bypass on the $L_i$ with a negative sign as shown in Figure \ref{dc}  B, C. Say the bypass is on $L_1$. We can attach this bypass and thicken our solid torus to get a solid torus with boundary slope $0$. Hence we have a toric annulus $L_1$ with boundary slopes $0$ and $\infty$ and an extension of this toric annulus is another toric annulus with boundary slopes $\infty$ and $0$. Hence we have too much radial twisting in $V_1'$ and hence the contact structures on $V_1'$ are overtwisted. Similarly, we get an overtwisted contact structure if we have a boundary-parallel dividing arc on boundary component two. Hence we are left with the case where both $L_1$ and $L_2$ have negative signs as shown in Figure \ref{dc} D. This contact structure on $\Sigma\times S^1$ is denoted by $\xi_E$. We would get the opposite signed configuration if we had started Case 2B with the other sign configuration. This will be denoted $\xi_E'$ and it is shown in Figure \ref{E'}. 
\begin{figure}
   \centering
    \includegraphics[width=3in]{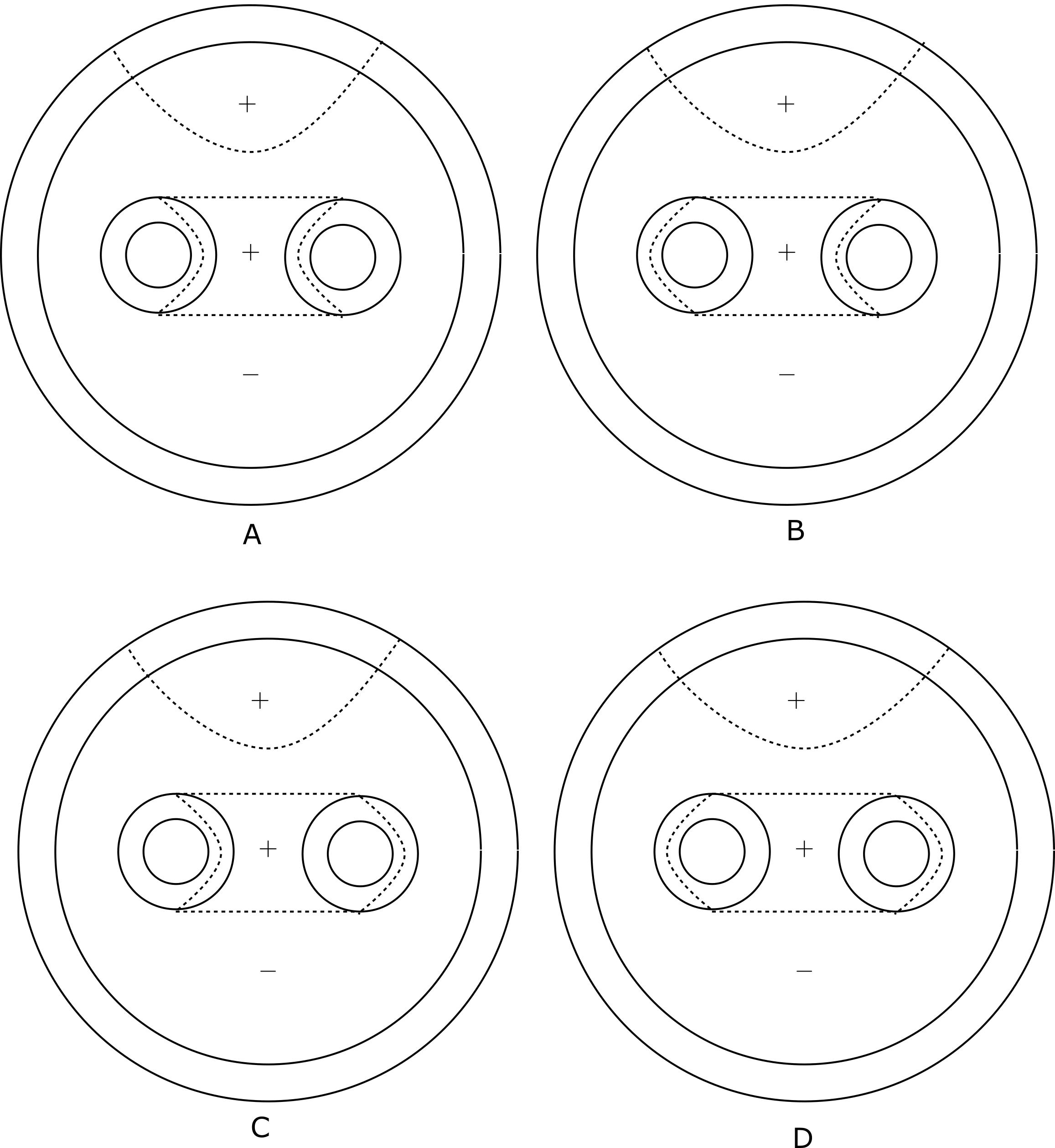}
     \caption{Four possible configurations of dividing curves on $\Sigma$.}
    \label{dc}
\end{figure}\ 

\begin{figure}
   \centering
    \includegraphics[width=2in]{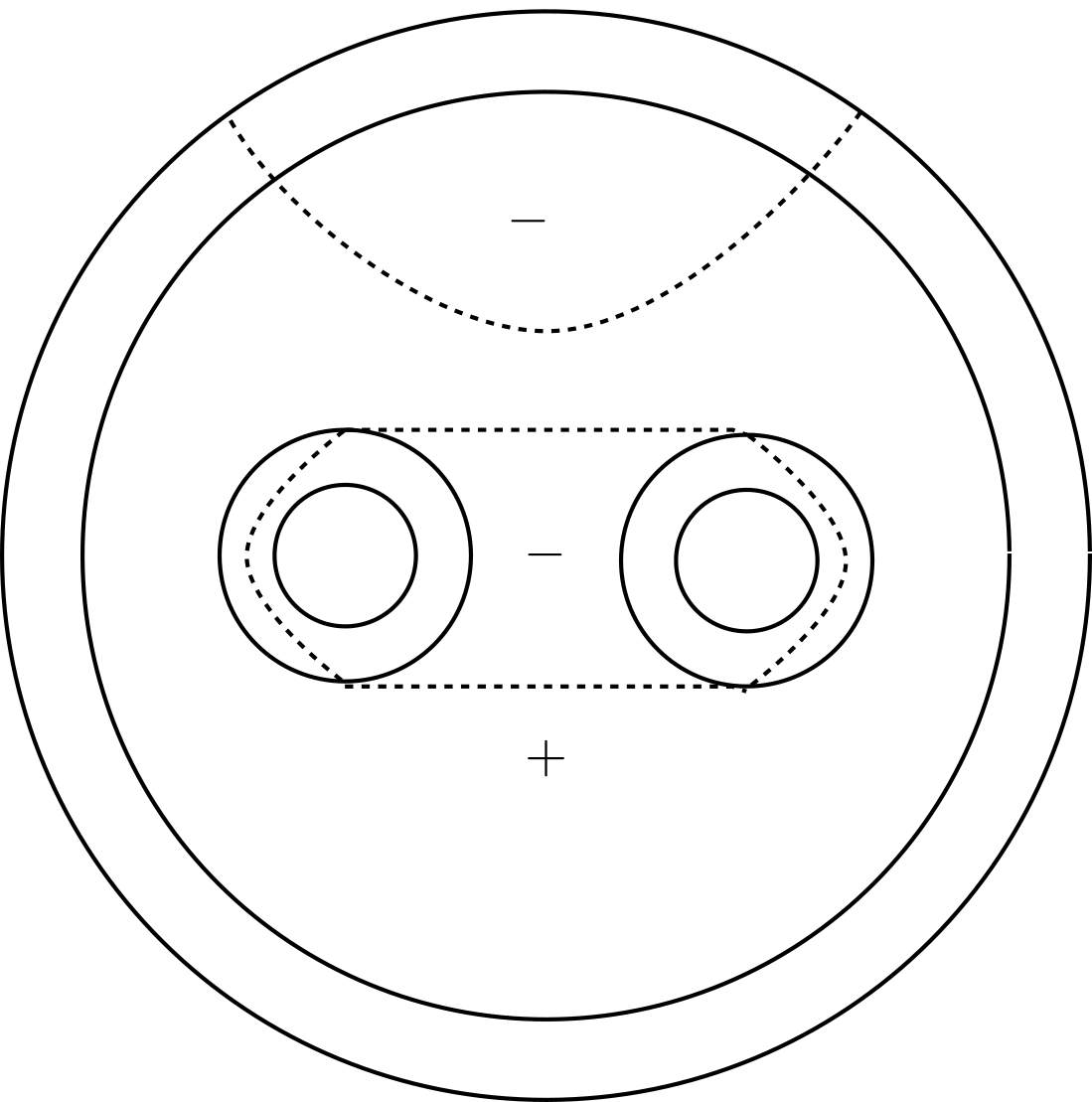}
     \caption{Dividing set of $\xi_{E'}$ on $\Sigma$.}
    \label{E'}
\end{figure}\

We have nine tight contact structures on $\Sigma\times S^1$ denoted by $\xi_A$, $\xi_B$, $\xi_{B'}$, $\xi_C$, $\xi_{C'}$, $\xi_D$, $\xi_{D'}$, $\xi_E$, $\xi_{E'}$. There is a unique tight contact structure on each $V_i$. Let us glue the two $\Sigma \times S^1$ along the toric annulus (refer section \ref{decompose}). When we are gluing the two $\Sigma\times S^1$ we are gluing using an orientation reversing diffeomorphism so we have to glue them using $T^2\times I$ with boundary slopes $-1$ and $+1$. After gluing, we get \{sphere with four holes\}$\times S^1$, denoted by $X\times S^1$, which we call $Y$. Let us look at all possible tight contact structures on $Y$ with zero Giroux torsion. Figure \ref{Combi1} shows all possible dividing curve configurations on the sphere with four holes which give potentially tight contact structures with zero Giroux torsion after gluing the two pairs of pants. The tight contact structure on $Y$ we get by gluing two pairs of pants with $\xi_E$\ and $\xi_{E'}$ is shown in Figure \ref{Combi2}. This contact structure has non-zero Giroux torsion depicted by the dividing curves on the toric annulus. They describe the twisting of contact planes on toric annulus $\times S^1$ from slope $0$ to $\infty$ to $0$ to $\infty$ to $0$, to $ \ infty$, which is one full twist. All other combinations of gluing two pairs of pants give us an obvious overtwisted disk in $Y$.  

\ 
 
When we glue the four solid tori to $Y$ with contact structure as determined by the dividing curves in pictures 5, 6, 7, and 8 of Figure \ref{Combi1} we get overtwisted contact structures on $M(0;-1/2,-1/2,-1/2,-1/2)$, indeed in all of these cases, there is a boundary parallel dividing curve, say $\gamma_1$, on the $X$. There is a boundary parallel torus, say $T_1$ containing $\gamma_1$. We denote this boundary torus as $T_0$ and the dividing curve on it as $\gamma_0$. 
Consider the $T^2\times [0,1]$ from $T_0$ to $T_1$. The dividing curves $\gamma_0$ and $\gamma_1$ have slope zero, hence we have a boundary parallel torus corresponding to every slope in this $T^2\times I$. So in particular we have a torus with slope $\frac{1}{2}$. Using $A_i^{-1} $ this corresponds to slope zero in the basis of the glued in solid torus. Hence this dividing curve bounds a meridional disk in the solid torus which is our overtwisted disk.  \ 

\ 

We use the relative Euler class (see Section 4.2 in \cite{Hon01}) to show that pictures 3 and 4 (as shown in Figure \ref{Combi1}) yield non-isotopic tight contact structures on $Y$. We denote the $i$th tight contact structure in Figure \ref{Combi1} as $\xi_i$. We have that the relative Euler class $e(\xi_3)=-2$ whereas $e(\xi_4)=2$. 
Diagrams 1 and 2 (as shown in Figure \ref{Combi1}) can be shown to be equivalent by section changes similarly to those discussed in Section \ref{gc}. So we get at most three potentially tight contact structures on $X \times S^1$.   \

\begin{figure}
   \centering
    \includegraphics[width=5in]{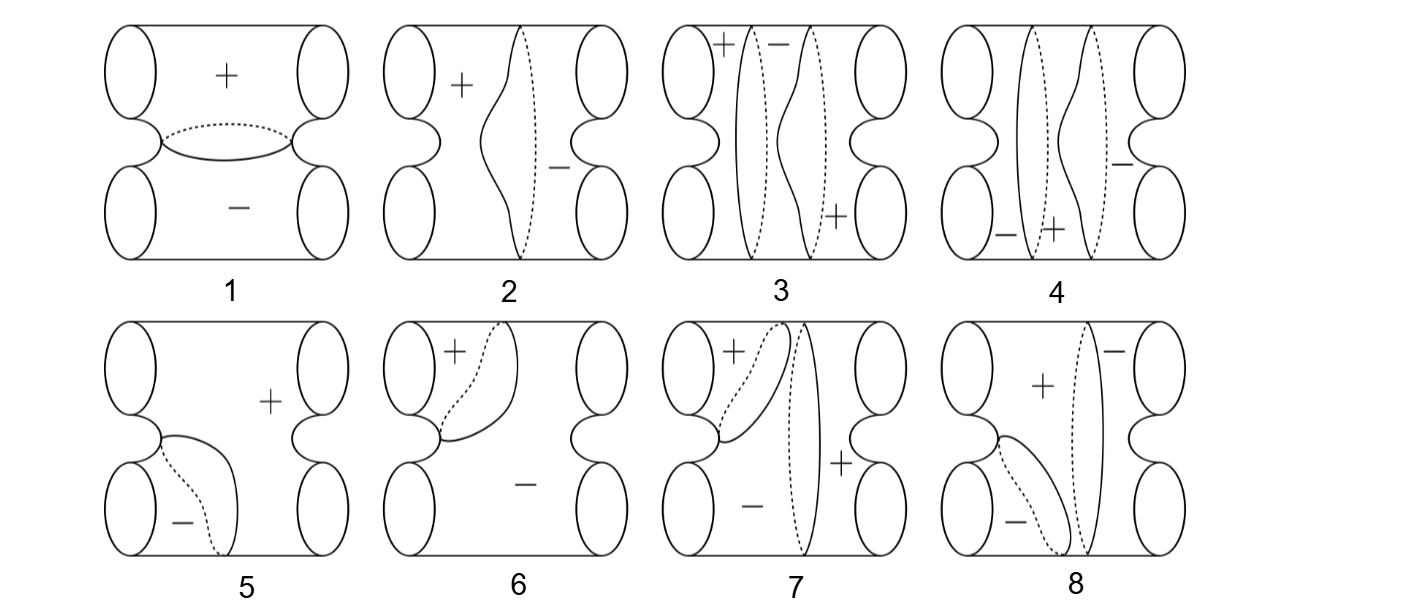}
     \caption{Dividing curve configurations on the sphere with four holes representing all possible tight contact structures on $Y$ with zero Giroux torsion.}
    \label{Combi1}
\end{figure}\

\begin{figure}
   \centering
    \includegraphics[width=2in]{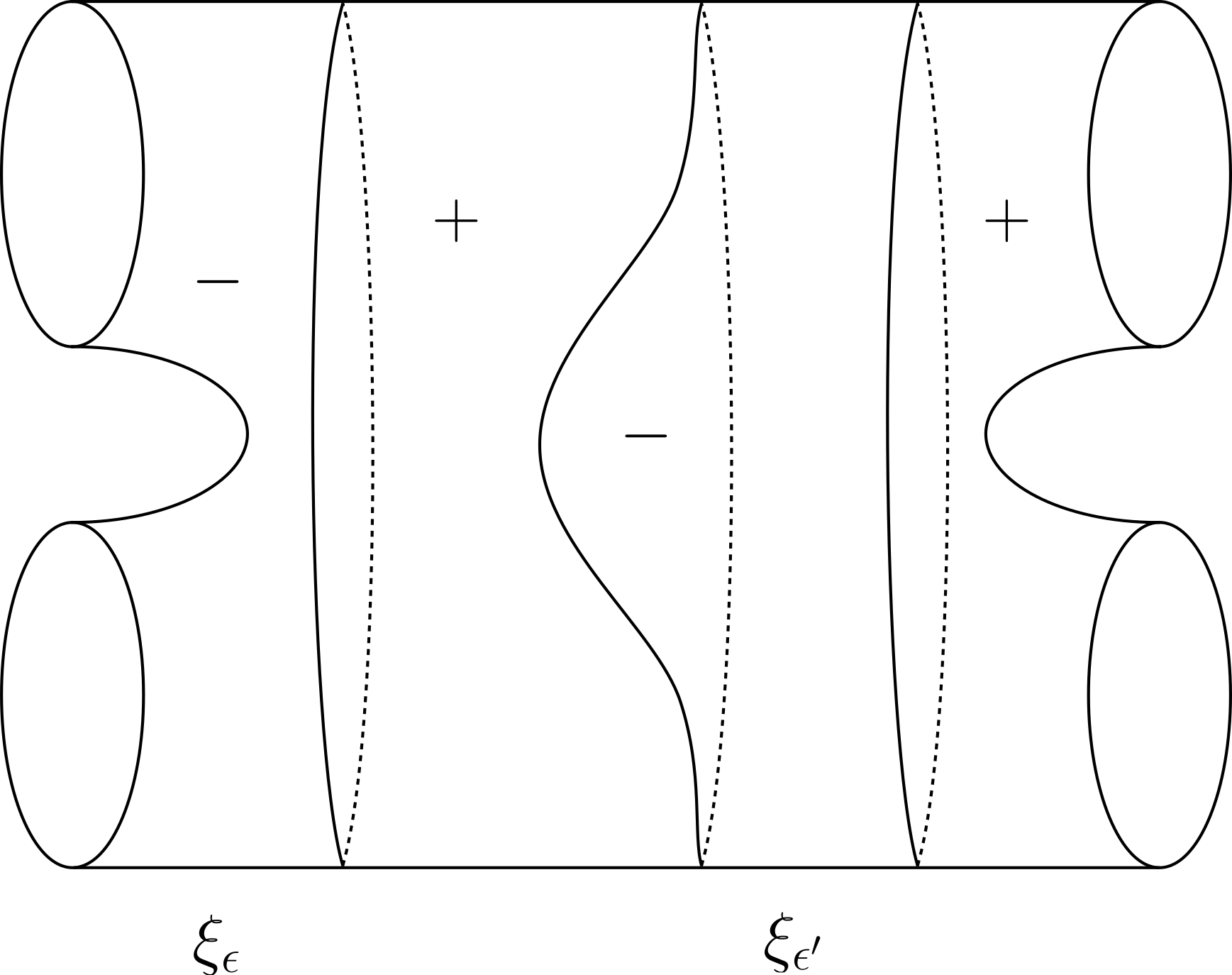}
     \caption{Dividing curve configuration showing non-zero Giroux torsion.}
    \label{Combi2}
\end{figure}\ 

\label{basic} Since the upper bound and lower bound (which we computed in Section \ref{Example}) both are three, we get that $|\pi_0(\Tight^{min}(M))|=3$. All of these are Stein fillable, since we get these three contact structures on $M$ by doing a $-1$ Legendrian surgery on a Legendrian link in ($S^3, \xi_{std}$). \ 

\ 


 Since $M$ is toroidal there are $\mathbb{Z}$ many non-isotopic tight contact structures on $M$ corresponding to integral Giroux torsion. It is proved by Gay in \cite{Gay} that tight contact structures that are strongly symplectically fillable have no Giroux torsion. Also, it is proved in \cite{Gei} that a weakly fillable tight contact structure on a rational homology sphere is a strongly fillable contact structure. The Seifert fiberd manifold $M$ is a rational homology sphere. Hence the tight contact structures on $M$ with non-zero Giroux torsion are not weakly fillable.
 \end{proof}
\subsection{General case}
\label{gc}

Consider the manifold $M=M(0,e_0;p_1/q_1,...,p_4/q_4)$ where $e_0\leq -4$ and $p_i, q_i \in \mathbb{Z}$ with $e_0\in \mathbb{Z}$, $\frac{p_i}{q_i}\in (0,1)\cap \mathbb{Q}$, and $(p_i,q_i)=1$, $\frac{-q_i}{p_i}=[a_0^i,a_1^i,...,{a_{m_i}}^{i}]$, where all  $a_j^i$'s are integers, $a_0^i=\lfloor \frac{-q_i}{p_i} \rfloor \leq -1,$ and $a_j^i\leq -2$ for $j\geq 1$. We define $p_j^i={-a_j^i}p_{j-1}^i-p_{j-2}^i$ for $j=0,1,...,m_i$ and $p_{-2}^i=-1$ and $p_{-1}^i=0$. Similarly we define $q_j^i=-a_j^i q_{j-1}^i-q_{j-2}^i$ for $j=0,1,...,m_i$ and $q_{-2}^i=-1$ and $q_{-1}^i=0$. The previous example (Section \ref{Example}) we have considered has $e_0(M)=-4$ and $\frac{-q_i}{p_i}=-2$. Once we have calculated the tight contact structures on this example case, it is computationally easy to generalise to the case of all manifolds $M$ with $e_0(M)\leq-4.$\ 
\\


\begin{figure}
   \centering
    \includegraphics[width=3in]{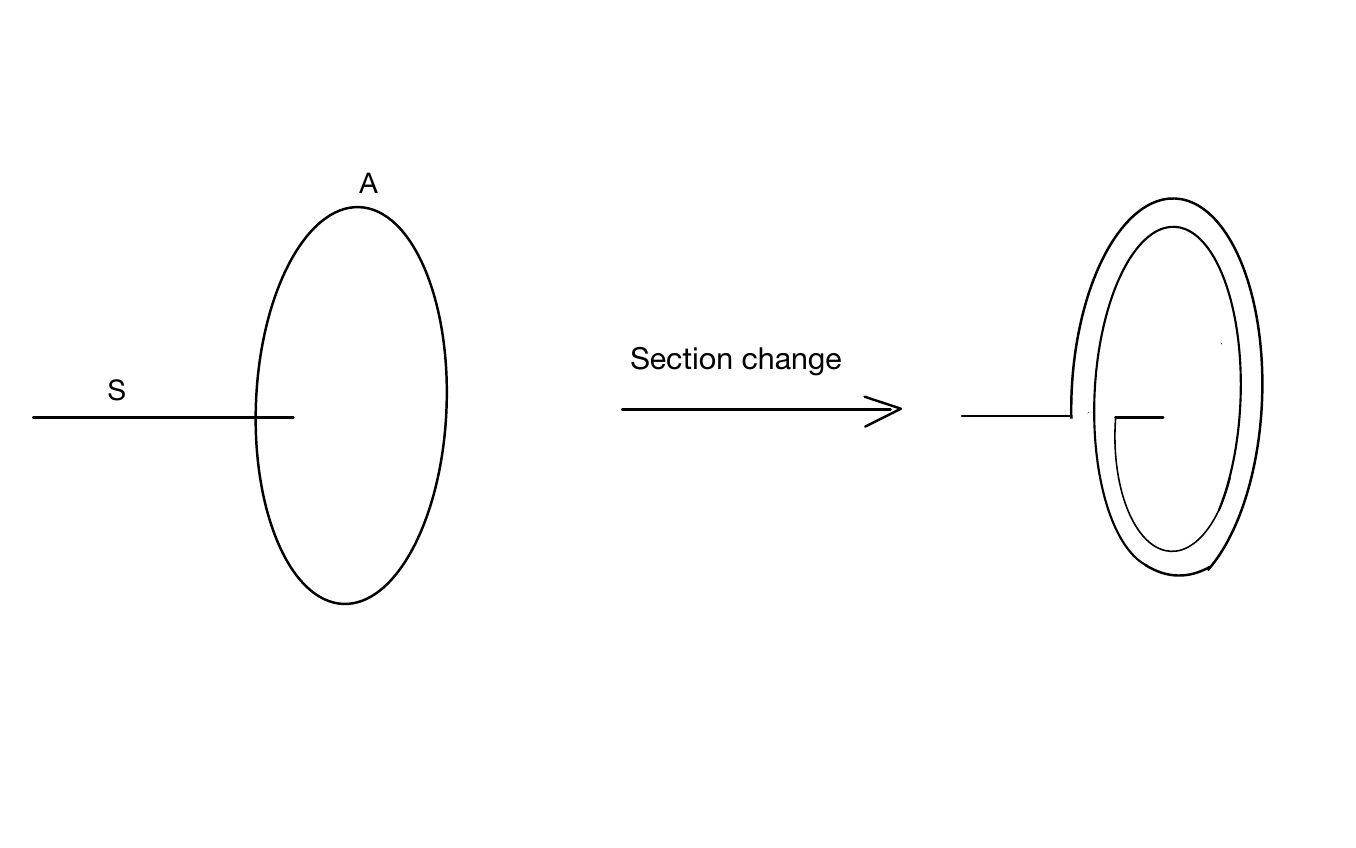}
   \caption{Section change to get diffeomorphic Seifert fibered manifold \cite{Hat}}
    \label{sechat}
\end{figure}

\begin{proof}[Proof of Theorem \ref{general thm}]
To prove this theorem we follow similar methods as we used for our example case (see Section \ref{Example}). We need an additional section change operation to show that some of the potential tight contact structures are isotopic. The operation that Hatcher (\cite[Proof of Proposition 4.1]{Hat}) demonstrates in Figure 4.1, is what we denote as section change, although Hatcher never uses this terminology. Here we briefly describe this operation for the reader's convenience. Since we have a circle bundle $M'\rightarrow B'$ we can choose a section. This choice determines the slope on the boundary of $M'$. Let's examine how the slopes on the boundary change when we opt for a different section. Say $a$ is an arc with endpoints in $\partial B'$. There is an annulus corresponding to this arc in $M'$, say $A$. The new section we choose winds $m$ times around $A$ as it crosses $A$ as illustrated in Figure \ref{sechat}. This effectively adds $m$ to the boundary slope at one end of $A$ and subtracts $m$ from the other boundary slope.\\ 

The surgery diagram for the construction of $M=M(0,e_0;p_1/q_1,...,p_4/q_4)$ is shown in Figure \ref{lg} on top left. We perform Rolfsen twists to get the surgery diagram on the top right. We do a slam dunk operation to obtain the diagram at the bottom. One can refer to \cite{GS} for both these operations. The number of Legendrian realisations for each of the unknots is $(a_i^j+1)$ or $(e_0(M)+1)$ based on their Thurston-Bennequin and rotation numbers. Then by Theorem \ref{LM} there are at least $|(e_0(M)+1)\Pi_{i=1}^4\Pi_{j=1}^{m_i}(a_j^i+1)|$ tight contact structures with zero Giroux torsion on $M$ up to contact isotopy. This gives a lower bound on the number of tight contact structures on $M$.
\begin{figure}
   \centering
    \includegraphics[width=6in]{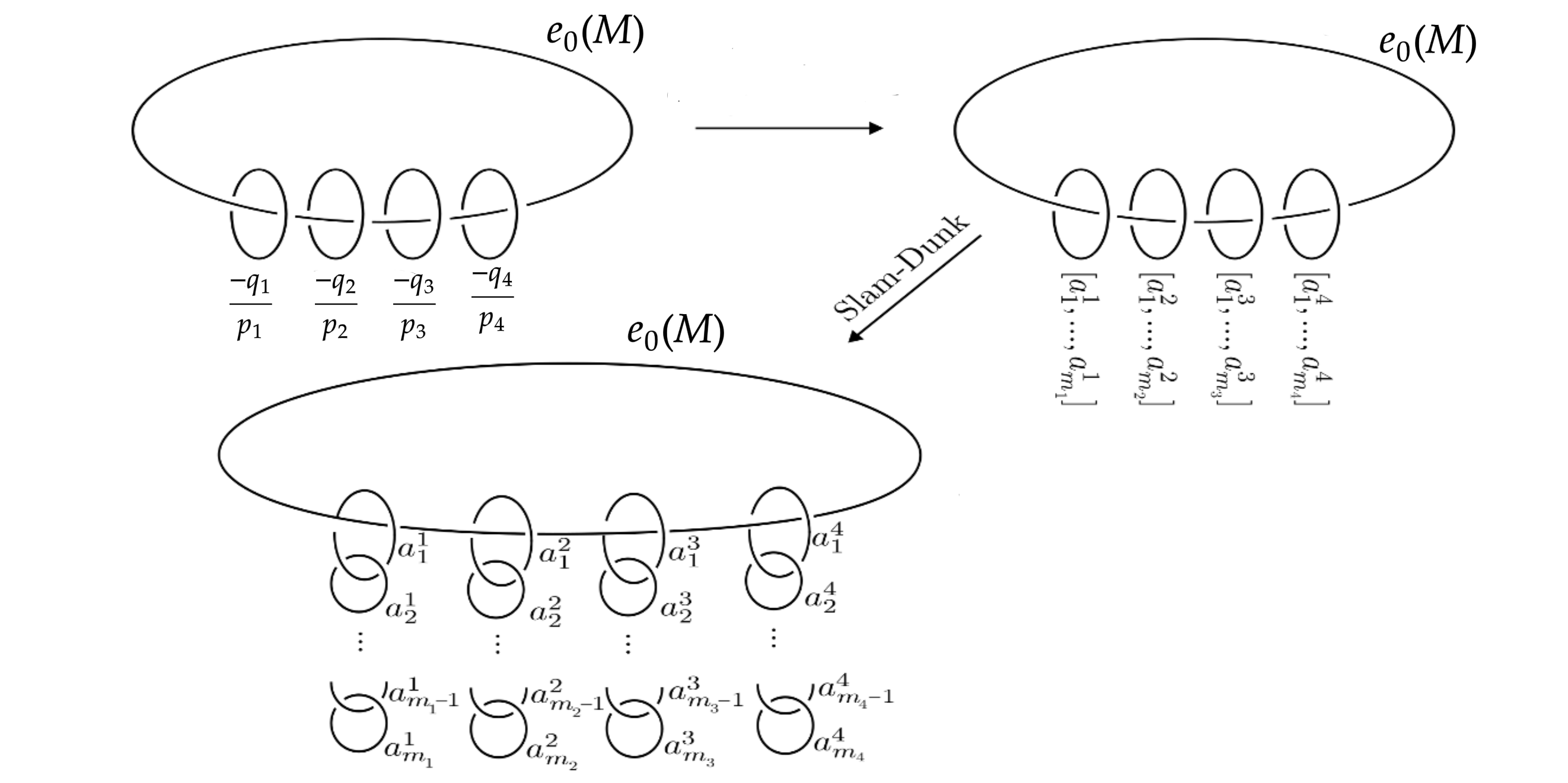}
   \caption{Surgery diagram representations of the manifold $M=M(0,e_0;p_1/q_1,...,p_4/q_4)$}
    \label{lg}
\end{figure}

\subsubsection{Upper Bound}

We start with the decomposition of the Seifert fibered manifold as stated previously in the example case. The manifold $M=M(0,e_0;p_1/q_1,...,p_4/q_4)$ has incompressible tori. We are going to start by counting tight contact structures with zero Giroux torsion on $M$. 

\ 


Using the same construction and notations from the example case as in Figure \ref{U1} we have that the two attaching maps $A_i:\partial V_i\rightarrow -\partial (\Sigma\times S^1)_i$ is given by 
$\left(\begin{array}{ll}
p_i & u_i\\
q_i & v_i
\end{array}\right)\in SL_2(\mathbb{Z})$ for $i=1,2$ where $u_i=p_{m_{i}-1}^i$ and $v_i=q_{m_{i}-1}^i$. Using the flexibility of Legendrian ruling we assume that the ruling slope of $\partial(\Sigma\times S^1)_{i}$ for $i=1,2,3$ is infinite. 
Assume that the fibers $F_i$ are simultaneously isotoped to Legendrian curves such that their twisting numbers are particularly negative. For $i=1,2$ we have $A_i.(n_i,1)^T=(n_ip_i+u_i,n_iq_i+v_i)$. We denote the slope of the dividing curve on $-\partial(\Sigma\times S^1)_i$ by $s_i=\displaystyle\frac{n_iq_i+v_i}{n_ip_i+u_i}=\displaystyle\frac{q_i}{p_i}+\displaystyle\frac{1}{p_i(n_ip_i+u_i)}$. \ 

\ 

Note that the manifolds, $M=M(0,e_0;p_1/q_1,...,p_4/q_4)$, we are working with are L-spaces (see Theorem 1.1 \cite{LS}). We look at Seifert fibered manifold $M'=(0;-q_1/p_1,...,-q_3/p_3)$ which is an L-space with $e_0\leq-2$. Any tight contact structures on $M'$ has a Legendrian curve with twisting number $-1$ in the $\Sigma \times S^1$ (Using Corollary 5.2 from \cite{G}). We call this curve $L$ and its neighborhood $V$. We can construct $M$ from $M'$ by doing a surgery on a fiber which is not in $V$. Hence we have a Legendrian curve with twisting number $-1$ in our manifold $M$. We can assume that the Legendrian ruling slope on $\partial(\Sigma\times S^1)_{i}$ for $i=1,2$ is $\infty$. Take an annulus $A_i$ between $\partial V$ and $\partial(\Sigma\times S^1)_{i}$ for $i=1,2$. There might be some bypasses on $\partial(\Sigma\times S^1)_{i}$ by the imbalance principle. After attaching these bypasses the slope of the dividing curve on $\partial(\Sigma\times S^1)_{i}$ is $\lfloor\frac{q_i}{p_i}\rfloor$ for $i=1,2$. One can similarly prove that  the dividing curve on $\partial(\Sigma\times S^1)_{i}$ is $\lfloor\frac{q_i}{p_i}\rfloor$ for $i=3,4$.\ 

\ 



\

\subsubsection{Combining tight contact structures on the basic blocks}
After we glue the two $\Sigma\times S^1$ across the toric annulus we get $X\times S^1$ which we call $Y'$. The slope of the dividing curves on the boundary tori of $Y'$ is $\lfloor \frac{q_i}{p_i} \rfloor $ for $i=1,2,3,4$. We take a suitable diffeomorphism of $X\times S^1$ to normalize the boundary slopes to be $\Sigma_{i=1}^4 \lfloor \frac{q_i}{p_i} \rfloor = s$ on one of the boundary torus and $0$ on the other three boundary tori. This amounts to a section change. We denote this $X\times S^1$ by $Y$. The number of tight contact structures up to contact isotopy on the manifold before the section change is the same as the number of tight contact structures up to contact isotopy on the manifold after the section change. Since $s\geq 0$ we can decompose $Y$ in a $X\times S^1$ with all four boundary slopes $0$ and a toric annulus with boundary slopes $0$ and $s$. This is shown in Figure \ref{Y1}.

 \begin{figure}
   \centering
    \includegraphics[width=6in]{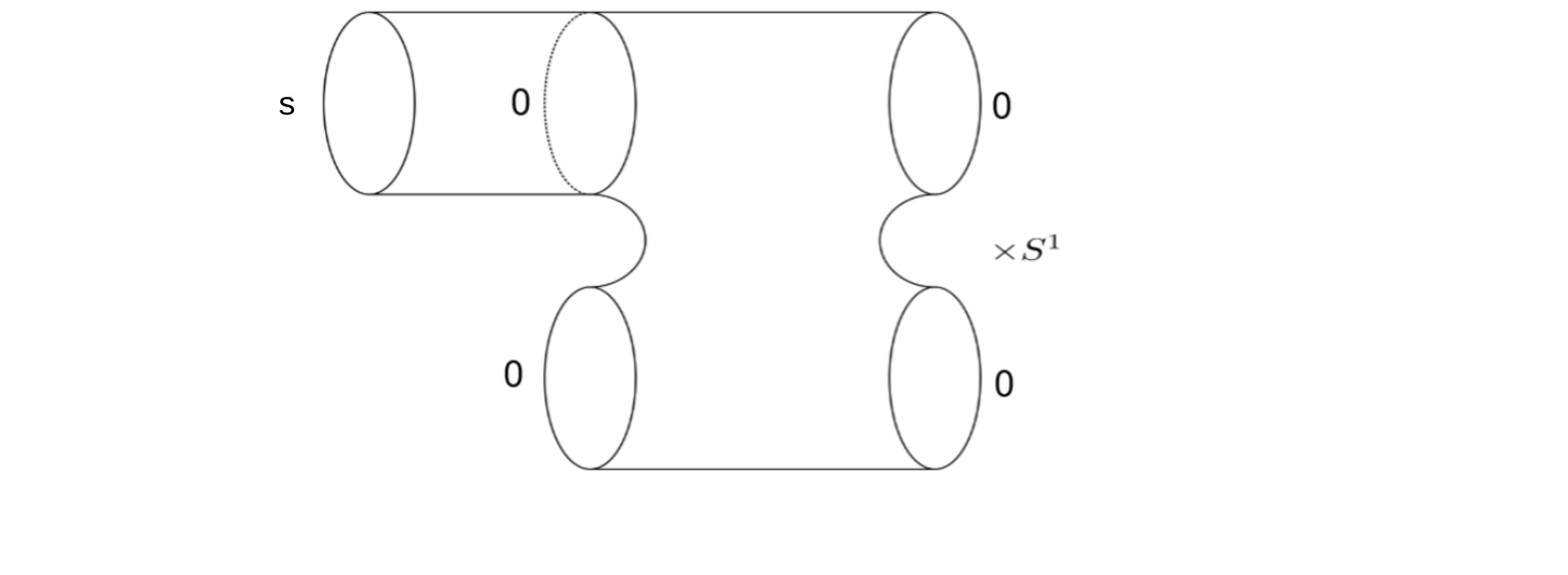}
     \caption{Decomposition of Y.}
    \label{Y1}
\end{figure}

\ 

There are three tight contact structures with zero Giroux torsion on $X\times S^1$ with all four boundary slopes $0$ (as proved in Section \ref{basic}). There are $(s+1)$ tight contact structures on a toric annulus with boundary slopes $0$ and $s$ \cite{Hon01}. When we glue this toric annulus on $X\times S^1$, we need the signs of the regions on the boundary to match. There are $3s+3$ combinations but only $2s+3$ are possible because the signs of the regions should match at the gluing. So we get an upper bound of $2s+3$ on the number of potential tight contact structures. 
 \ 

\ 

We write the boundary slopes on $Y$ as $(s_1, s_2, s_3, s_4)$ where the slope $s_i$ is $s(\partial Y)_i$. Say that our boundary slopes on Y are $(2,0,0,0)$. We will show that the two dividing curve configurations shown on a $X$ in Figure \ref{sc1} (top right and top left) represent two different sections (as in Prop. 2.1 in \cite{Hat}) of the same contact structure on $X\times S^1$. We start with a tight contact structure that has a positive and negative bypass on the boundary component with boundary slope $2$. This dividing curve configuration is shown on the top left. Consider an annulus $A$ between boundary component having slopes $2$ and $0$, as shown in Figure \ref{sc1}. We do the section change (as in Prop. 2.1 in \cite{Hat}) across annulus $A$ to get $Y ''$ with boundary slopes $(1,1,0,0)$. We obtain the dividing curves on the new section of $Y''$ as follows: we start with our old section on $X\times S^1$ with boundary slopes $(2,0,0,0)$ and draw the dividing curves on the $X$ as well as on all four boundary tori. Using the construction as in Prop. 2.1 in \cite{Hat} we draw our new section (with boundary slopes $(1,1,0,0)$) in our old $X\times S^1$ (with boundary slopes $(2,0,0,0)$). We mark all intersections of the dividing curves with the boundary of the new section. Since the new $X$ is embedded in old $X\times S^1$, we start drawing the dividing curve at one of the boundary components with slope $1$ and follow through in our new $X$ across the old section. We get the dividing curves connecting the two boundary components with boundary slopes $1$. This dividing curve configuration is shown in Figure \ref{sc1} at the bottom. Next, we look at the dividing curve configuration which has the two same signed bypasses on the boundary component with slope $2$ and half a twist across the incompressible torus. This dividing curve configuration is shown on the top right. We do the section change along annulus $A'$ between boundary components having slopes $2$ and $0$ which goes across the half twist. We get the dividing curves (using the same process as above) connecting the two boundary components with boundary slopes $1$. After the section change, we get the dividing curve configuration as shown in Figure \ref{sc1} at the bottom. Hence one can get the top right dividing curve configuration from the top left dividing curve configuration by section change on the $X\times S^1$. Hence the top right and the top left dividing curve configurations are two different sections of the same tight contact structure on $Y$.

 \begin{figure}
   \centering
    \includegraphics[width=2.5in]{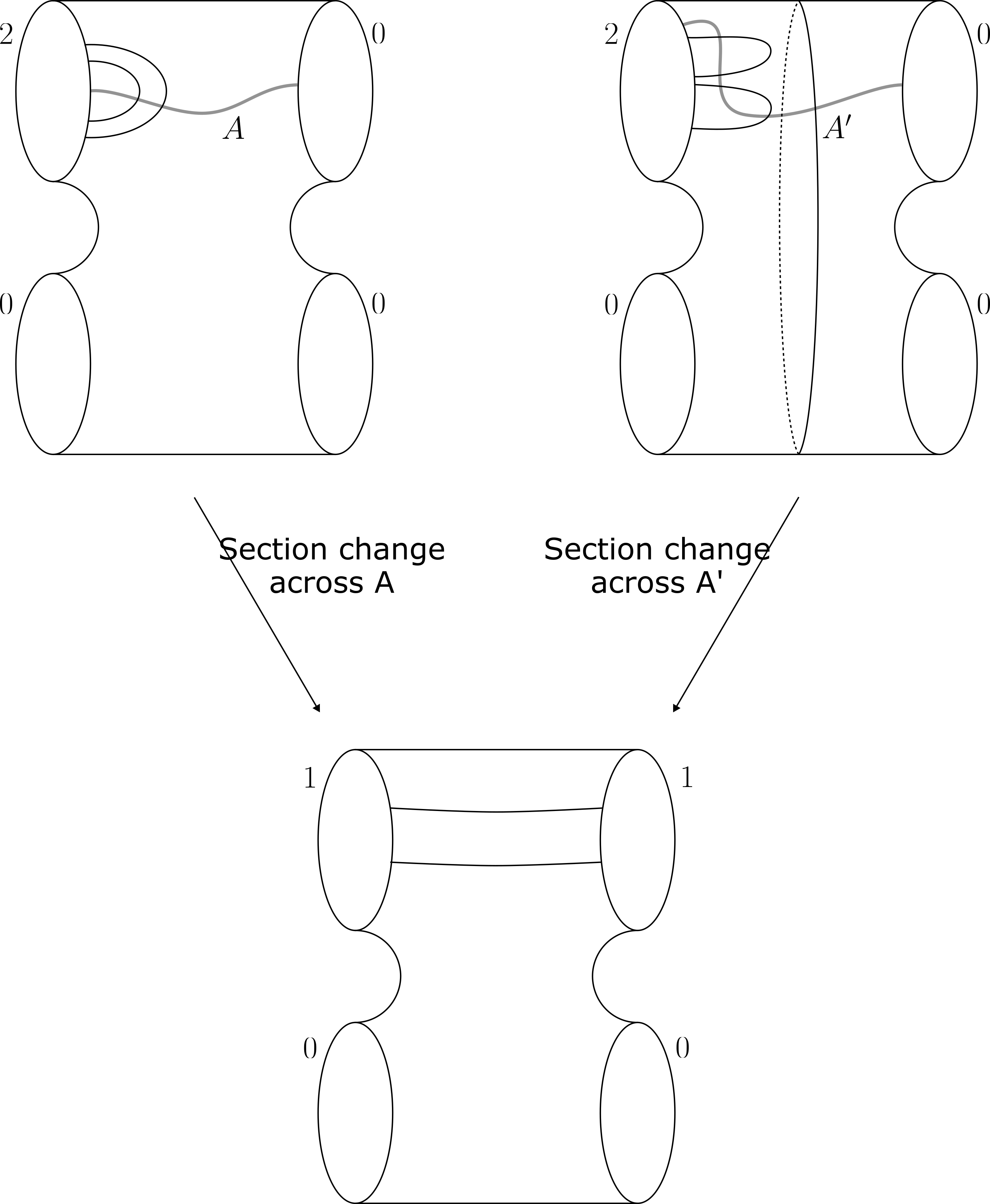}
     \caption{Different sections of a tight contact structure on $X\times S^1$.}
    \label{sc1}
\end{figure}

\ 

We had the upper bound $2s+3$ of potential tight contact structures. By the section change described above, $s$ many tight structures are counted twice: one represented by a section with a positive and a negative bypass and the other section with the same signed bypass and half a twist across the incompressible torus. Hence we get a tighter upper bound of $s+3$. Since $s=\Sigma_{i=1}^4 \lfloor \frac{q_i}{p_i} \rfloor$ and we get that $s+3=|e_0(M)+1|$. Hence the number of the potential tight contact structures on $X\times S^1$ is bounded above by $|e_0(M)+1|$.


\ 

The solid torus $V_i$ has a boundary slope of $-\frac{q_i-\lfloor\frac{q_i}{p_i}\ \rfloor p_i}{v_i-\lfloor\frac{q_i}{p_i}\rfloor u_i}$ = $-\frac{q_i + (a_0^i +1) p_i}{v_i+ (a_0^i) u_i}$. With this boundary slope, there are exactly $\Pi_{j=1}^{m}(a_j^i+1)$ tight contact structures on $V_i$ (Theorem 2.3 in \cite{Hon01}). Thus, up to contact isotopy there are at most $|(e_0(M)+1)|\Pi_{i=1}^4\Pi_{j=1}^{m}(a_j^i+1)$ tight contact structures with zero Giroux torsion on $M=M(0;-q_1/p_1,-q_2/p_2,-q_3/p_3,\allowbreak -q_4/p_4)$ with $e_0(M)\leq -4$. 

\ 

Since the upper bound and lower bound match, we get $|\pi_0(\Tight^{min}(M))|=|(e_0(M)+1)|\Pi_{i=1}^4\Pi_{j=1}^{m}(a_j^i+1).$ All of these are Stein fillable since we get these contact structures on $M$ by doing a $-1$ Legendrian surgery on a Legendrian link in ($S^3, \xi_{std}$). \

\

Since $M$ is toroidal there are $\mathbb{Z}$ many non-isotopic tight contact structures on $M$ corresponding to integral Giroux torsion. As seen in the example case the tight contact structures on $M$ with non-zero Giroux torsion are not weakly fillable.
\end{proof}
  
  \subsection{Concluding remarks}
  One can try to classify tight contact structures on Seifert fiberd manifold $M=M(0,e_0;p_1/q_1,...,p_4/q_4)$ where $e_0> -4$ and with,  $(p_i,q_i)=1$. We can construct tight contact structures with zero Giroux torsion by Legendrian surgery to get a lower bound on the number of tight contact structures. To get the upper bound on the number of tight contact structures we use convex surface theory. These two bounds don't match. Currently, the author is unable to find any contact isotopy between the contact structures found by convex surface theory or find an invariant to say that those are non-isotopic contact 
  
\bibliographystyle{plain}
\bibliography{references.bib}
\vspace{10pt}

\author{Tanushree Shah} \\
\email{tanushrees@cmi.ac.in}\\
\address{Chennai Mathematical Institute}
\end{document}